\documentclass[12pt]{amsart}

\usepackage{amsmath}
\usepackage{amsthm}
\usepackage{amsfonts}
\usepackage{dsfont}
\usepackage{enumitem}
\usepackage{color}
\usepackage{fullpage}

\newcommand{\Yp}{Y_{+}}

\newcommand{\Ym}{Y_{-}}
\newcommand{\Yz}{Y_{z}}

\newcommand{\complex}{\mathbb{C}}
\newcommand{\naturals}{\mathbb{N}}
\newcommand{\integers}{\mathbb{Z}}

\newcommand{\angles}[1]{\left\langle #1 \right\rangle}

\newcommand{\para}[1]{\left(#1\right)}
\newcommand{\paraa}[1]{\big(#1\big)}
\newcommand{\parab}[1]{\Big(#1\Big)}
\newcommand{\parac}[1]{\bigg(#1\bigg)}

\newcommand{\spacearound}[1]{\quad#1\quad}
\newcommand{\equivalent}{\spacearound{\Leftrightarrow}}
\renewcommand{\implies}{\spacearound{\Rightarrow}}
\newcommand{\qtext}[1]{\quad\text{#1}\quad}

\newcommand{\qand}{\qtext{and}}

\newtheorem{theorem}{Theorem}[section]

\newtheorem{lemma}[theorem]{Lemma}
\newtheorem{proposition}[theorem]{Proposition}

\theoremstyle{definition}
\newtheorem{definition}[theorem]{Definition}
\theoremstyle{remark}
\newtheorem{remark}[theorem]{Remark}
\numberwithin{equation}{section}

\renewcommand{\mid}{\mathds{1}}
\newcommand{\Xp}{X_{+}}
\newcommand{\Xpm}{X_{\pm}}

\newcommand{\Xm}{X_{-}}
\newcommand{\Xz}{X_{z}}
\newcommand{\omegap}{\omega_{+}}
\newcommand{\omegam}{\omega_{-}}
\newcommand{\omegaz}{\omega_{z}}
\newcommand{\Sthreeq}{S^{3}_{q}}
\newcommand{\Stq}{\Sthreeq}

\newcommand{\TSthreeq}{T\Sthreeq}
\newcommand{\TStq}{\TSthreeq}
\newcommand{\Stwoq}{S^{2}_{q}}
\newcommand{\Uqsu}{\mathcal{U}_q(\textrm{su}(2))}

\newcommand{\A}{\mathcal{A}}

\newcommand{\lt}{\triangleright}
\newcommand{\rt}{\triangleleft}
\newcommand{\one}[1]{#1_{(1)}}
\newcommand{\two}[1]{#1_{(2)}}

\newcommand{\nablat}{\tilde{\nabla}}
\newcommand{\nablap}{\nabla_{+}}
\newcommand{\nablam}{\nabla_{-}}

\newcommand{\nablaz}{\nabla_{z}}
\newcommand{\nzero}{\nabla^0}

\newcommand{\OmegaStq}{\Omega^1(\Sthreeq)}

\newcommand{\Gammat}{\widetilde{\Gamma}}

\renewcommand{\d}{\partial}

\newcommand{\rtr}{\triangleright}
\newcommand{\ltr}{\triangleleft}
\newcommand{\Mat}{\operatorname{Mat}}
\newcommand{\htilde}{\tilde{h}}

\newcommand{\sigmap}{\sigma_{+}}
\newcommand{\sigmam}{\sigma_{-}}
\newcommand{\sigmaz}{\sigma_{z}}
\newcommand{\sigmapm}{\sigma_{\pm}}
\newcommand{\sigmas}{\sigma^{\ast}}
\newcommand{\sigmasp}{\sigmas_{+}}
\newcommand{\sigmasm}{\sigmas_{-}}
\newcommand{\sigmaspm}{\sigmas_{\pm}}
\newcommand{\sigmasz}{\sigmas_{z}}
\newcommand{\sigmah}{\hat{\sigma}}

\newcommand{\sigmahpm}{\sigmah_{\pm}}
\newcommand{\sigmahz}{\sigmah_{z}}
\newcommand{\sigmahs}{\sigmah^{\ast}}
\newcommand{\sigmahsp}{\sigmah^\ast_{+}}
\newcommand{\sigmahsm}{\sigmah^\ast_{-}}
\newcommand{\sigmahsz}{\sigmah^\ast_{z}}

\newcommand{\Bp}{B_+}
\newcommand{\Bm}{B_-}
\newcommand{\Bz}{B_0}
\renewcommand{\L}{\mathcal{L}}
\newcommand{\eh}{\hat{e}}

\newcommand{\gone}{\gamma^{(1)}}
\newcommand{\gtwo}{\gamma^{(2)}}
\newcommand{\gthree}{\gamma^{(3)}}
\newcommand{\thalf}{\tfrac{1}{2}}

\newcommand{\g}{\mathfrak{g}}

\title{On $q$-deformed Levi-Civita connections}
\date{5 May 2020}
\author{Joakim Arnlind, Kwalombota Ilwale and Giovanni Landi}

\address[Joakim Arnlind]{Dept. of Math.\\
Link\"oping University\\
581 83 Link\"oping\\
Sweden}
\email{joakim.arnlind@liu.se}

\address[Kwalombota Ilwale]{Dept. of Math.\\
Link\"oping University\\
581 83 Link\"oping\\
Sweden}
\email{kwalombota.ilwale@liu.se}

\address[Giovanni Landi]
{Matematica, Universit\`a di Trieste, 
\newline \indent Via A. Valerio, 12/1, 34127  Trieste, Italy 
\newline \indent Institute for Geometry and Physics (IGAP) Trieste, Italy 
and INFN, Trieste, Italy
}
\email{landi@units.it}

\subjclass[2000]{}
\keywords{}

\begin{document}

\maketitle

\begin{abstract}
  We explore the possibility of introducing $q$-deformed connections
  on the quantum 2-sphere and 3-sphere, satisfying a twisted Leibniz'
  rule in analogy with $q$-deformed derivations. We show that such
  connections always exist on projective modules.  Furthermore, a
  condition for metric compatibility is introduced, and an explicit
  formula is given, parametrizing all metric connections on a free
  module. For the module of 1-forms on the quantum 3-sphere, a
  $q$-deformed torsion freeness condition is introduced and we derive
  explicit expressions for the Christoffel symbols of a Levi-Civita
  connection for a general class of metrics satisfying a certain
  reality condition. Finally, we construct metric connections on a
  class of projective modules over the quantum 2-sphere.
\end{abstract}

\tableofcontents
\parskip = .55 ex

\section{Introduction}

\noindent
In recent years, a lot of progress has been made in understanding the
Riemannian aspects of noncommutative geometry. These are not only
mathematically interesting, but also important in physics where
noncommutative geometry is expected to play a key role, notably in a
theory of quantum gravity. In Riemannian geometry the Levi-Civita
connection and its curvature have a central role, and it turns
out that there are several different ways of approaching these objects
in the noncommutative setting (see e.g.
\cite{cff:gravityncgeometry,dvmmm:onCurvature,m:nc.spin.q-sphere,ac:ncgravitysolutions,bm:starCompatibleConnections,r:leviCivita,aw:curvature.three.sphere,bgm:levi-civita.class.spectral.triples,bgl2020}).

From an algebraic perspective, the set of vector fields and the set of
differential forms are (finitely generated projective) modules over
the algebra of functions, a viewpoint which is also adopted in
noncommutative geometry. However, considering vector fields as
derivations does not immediately carry over to noncommutative
geometry, since the set of derivations of a (noncommutative) algebra
is in general not a module over the algebra but only a module over the
center of the algebra. Therefore, one is lead naturally to focus on
differential forms and define a connection on a general module as
taking values in the tensor product of the module with the module of
differential forms. More precisely, let $M$ be a (left) $\A$-module
and let $\Omega^1(\A)$ denote a module of differential forms together
with a differential $d:\A\to\Omega^1(\A)$. A connection on $M$ is a
linear map $\nabla:M\to\Omega^1(\A)\otimes M$ satisfying a version of
Leibniz rule
\begin{align}\label{eq:diff.form.Leibniz}
  \nabla(fm) = f\nabla m + df\otimes m
\end{align}
for $f\in\A$ and $m\in M$. In differential geometry, for a vector
field $X$ one obtains a covariant derivative $\nabla_X:M\to M$, by
pairing differential forms with $X$ (as differential forms are
dual to vector fields). In a noncommutative version of the above,
there is in general no canonical way of obtaining a ``covariant
derivative'' $\nabla_X:M\to M$. In a derivation based approach to
noncommutative geometry (see
e.g. \cite{dv:calculDifferentiel,dvmmm:onCurvature}), one puts
emphasis on the choice of a Lie algebra $\g$ of derivations of the
algebra $\A$. Given a (left) $\A$-module $M$ one defines a connection
as a map $\nabla:\g\times M\to M$, usually writing
$\nabla(\d,m) = \nabla_\d m$ for $\d\in\g$ and $m\in M$, satisfying
\begin{align}
\nabla_{\d}(fm) = f\nabla_{\d}m + \d(f) \, m  
\end{align}
for $f\in\A$ and $m\in M$, in parallel with \eqref{eq:diff.form.Leibniz}.

For quantum groups, it turns out that natural analogues of vector
fields are not quite derivations, but rather maps satisfying a
twisted Leibniz rule. For instance, as we shall see, for the quantum
3-sphere $\Stq$ one defines maps $X_a:\Stq\to\Stq$ satisfying
\begin{align}\label{eq:Xp.def.Leibniz}
  X_a(fg) = f X_a(g) + X_a(f) \sigma_a(g)
\end{align}
for $f,g\in\Stq$, where $\sigma_a:\Stq\to\Stq$, for $a=1,2,3$, are
algebra morphisms.  Thus, in this note, we explore the possibility of
introducing a corresponding $q$-affine connection on a (left)
$\Stq$-module $M$. That is, motivated by \eqref{eq:Xp.def.Leibniz} we
introduce a covariant derivative $\nabla_{X_a}:M\to M$ such that
\begin{align}
  \nabla_{X_a}(fm) = f\nabla_{X_a}m + \Xp(f)\sigmah_a(m)
\end{align}
for $f\in\Stq$ and $m\in M$, where $\sigmah_a$ denotes an extension of
$\sigma_a$ to the module $M$ (cf. Section
\ref{sec:q.affine.connections}). In the following, we make these ideas
precise and prove that there exist $q$-affine connections on
projective modules. Furthermore, we introduce a condition for metric
compatibility, and in the particular case of a left covariant calculus
over $\Stq$, we investigate a derivation based definition of torsion.
Then we explicitly construct a Levi-Civita connection, that is
a torsion free and metric compatible connection. Moreover, we
construct metric connections on a class of projective modules over the
quantum 2-sphere. We note that the Riemannian geometry of quantum
spheres can be studied \cite{bm:starCompatibleConnections} from the
point of view of a bimodule connection on differential forms
satisfying \eqref{eq:diff.form.Leibniz} as well as a right Leibniz
rule twisted by a braiding map.

\section{The quantum 3-sphere}

\noindent
In this section we recall a few basic properties of the quantum
3-sphere \cite{w:twisted.su2}. The algebra $\Sthreeq$ is a unital $\ast$-algebra generated
by $a,a^\ast,c,c^\ast$ fulfilling
\begin{alignat*}{3}
&ac = qca &\qquad & c^*a^* =qa^*c^* &\qquad &ac^* =qc^*a \\
&ca^* =qa^*c & &cc^* = c^*c & &a^*a+c^*c =
aa^*+q^{2}cc^* =\mid 
\end{alignat*}
for a real parameter $q$. The identification of $\Sthreeq$ with the
quantum group $SEE_q(2)$ is via the Hopf algebra structure given by
\begin{alignat*}{2}
  &\Delta(a) = a\otimes a-qc^\ast\otimes c &\qquad
  &\Delta(c) = c\otimes a+a^\ast\otimes c\\
  &\Delta(a^\ast) = -qc\otimes c^\ast+a^\ast\otimes a^\ast &\qquad
  &\Delta(c^\ast) = a\otimes c^\ast+c^\ast\otimes a^\ast
\end{alignat*}
with antipode and counit
\begin{alignat*}{4}
  &S(a) = a^\ast &\qquad
  &S(c) = -qc &\qquad
  & \epsilon(a)=1 &\qquad
  & \epsilon(c)=0 \\
  &S(a^\ast) = a &\qquad
  &S(c^\ast) = -q^{-1}c^\ast &\qquad
  & \epsilon(a^\ast)=1 &\qquad
  & \epsilon(c^\ast)=0.
\end{alignat*}
Furthermore, the quantum enveloping algebra $\Uqsu$ is the $\ast$-algebra with generators
$E,F,K,K^{-1}$ satisfying
\begin{align*}
  &K^{\pm 1}E = q^{\pm 1}EK^{\pm 1}\qquad
    K^{\pm 1}F = q^{\mp 1}FK^{\pm 1}\qquad
  [E,F] = \frac{K^2-K^{-2}}{q-q^{-1}} .
\end{align*}
The corresponding Hopf algebra structure is given by the coproduct, 
\begin{align*}
  \Delta(E) = E\otimes K + K^{-1}\otimes E \qquad
  \Delta(F) = F\otimes K + K^{-1}\otimes F
  \qquad \Delta(K^{\pm 1}) = K^{\pm 1}\otimes K^{\pm 1}
\end{align*}
together with antipode and counit
\begin{alignat*}{3}
  &S(K) = K^{-1} &\qquad &S(E) = -qE &\qquad &S(F) = -q^{-1}F   \\
  &\epsilon(K) = 1 & &\epsilon(E) = 0 & &\epsilon(F) = 0.
\end{alignat*}
We recall that there is a unique bilinear pairing between $\Uqsu$ and
$\Sthreeq$ given by
\begin{alignat*}{2}
&\angles{K^{\pm 1}, a} = q^{\mp\, 1/2} &\qquad 
&\angles{K^{\pm 1}, a^*} = q^{\mp\, 1/2}\\
&\angles{E,c} = 1 &  &\angles{F,c^*} = -q^{-1}, 
\end{alignat*}
with the remaining pairings being zero. The pairing
induces a $\Uqsu$-bimodule structure on $\Sthreeq$ given by
\begin{align}\label{actions}
  h\lt f = \one{f}\angles{h,\two{f}}\qand
  f\rt h = \angles{h,\one{f}}\two{f}
\end{align}
for $h\in\Uqsu$ and $f\in\Sthreeq$ with Sweedler's notation
$\Delta(f)=\one{f}\otimes\two{f}$ (and implicit sum).  The
$\ast$-structure on $\Uqsu$, unconventionally denoted here by $\dag$
(for reasons that will become clear momentarily), is given by
$(K^{\pm 1})^\dag = K^{\pm 1}$ and $E^\dag = F$.  The action of
$\Uqsu$ is compatible with the $\ast$-algebra structures in the
following sense
\begin{align*}
  h\triangleright f^\ast = \paraa{S(h)^\dag\triangleright f}^\ast\qquad
  f^\ast\triangleleft h = \paraa{f\triangleleft S(h)^\dag}^\ast.
\end{align*}

Let us for convenience list the left and right actions of the
generators:
\begin{alignat*}{2}
  &K^{\pm 1}\triangleright a^{n}  = q^{\mp\frac{n}{2}}\,a^{n} &\qquad 
  &K^{\pm 1}\triangleright c^{n}  = q^{\mp\frac{n}{2}}\,c^{n}\\
  &K^{\pm 1}\triangleright a^{\ast}\,^{n} = q^{\pm\frac{n}{2}}(a^{\ast})^{n}&\qquad
  &K^{\pm 1}\triangleright c^{\ast}\,^{n} = q^{\pm\frac{n}{2}}(c^{\ast})^{n}\\
  &E\triangleright a^{n}  = -q^{(3-n)/2} [n]a^{n-1}c^{\ast} &\qquad
  &E\triangleright c^{n}  = q^{(1-n)/2}[n]c^{n-1}a^\ast\\
  &E\triangleright (a^{\ast})^{n}  = 0 &\qquad
  &E\triangleright (c^{\ast})^{n}  = 0.\\
  &F\triangleright a^{n}  = 0 &\qquad
  &F\triangleright c^{n}  = 0\\
  &F\triangleright (a^{\ast})^{n}  = q^{(1-n)/2}[n]c(a^{\ast})^{ n-1} &\qquad 
  &F\triangleright (c^{\ast})^{n}  = -q^{-(1+n)/2}[n]a(c^{\ast})^{n-1}
\end{alignat*}
and 
\begin{alignat*}{2}
  &a^n\ltr K^{\pm 1} = q^{\mp \frac{n}{2}}a^n &\quad
  &(a^\ast)^n\ltr K^{\pm 1} = q^{\pm \frac{n}{2}}(a^\ast)^n\\
  &c^n\ltr K^{\pm 1} = q^{\pm \frac{n}{2}}c^n &\quad
  &(c^\ast)^n\ltr K^{\pm 1} = q^{\mp \frac{n}{2}}(c^\ast)^n\\
  &a^n\ltr F = q^{\frac{n-1}{2}}[n]ca^{n-1} &\quad
  &(a^\ast)^n\ltr F = 0\\
  &c^n\ltr F = 0 &\quad
  &(c^\ast)^n\ltr F = -q^{\frac{n-3}{2}}[n]a^\ast(c^\ast)^{n-1}\\
  &a^n\ltr E = 0 &\quad
  &(a^\ast)^n\ltr E = -q^{\frac{n-3}{2}}[n]c^\ast (a^\ast)^{n-1}\\
  &c^n\ltr E = q^{\frac{n-1}{2}}[n]c^{n-1}a &\qquad
  &(c^\ast)^n\ltr E = 0
\end{alignat*}
where $[n] = (q^n-q^{-n})/(q-q^{-1})$.  

\subsection{$q$-deformed derivations}

The algebra $\Stq$ comes with a standard set of three $q$-deformed
derivations (which can be used to generate a left covariant
differential calculus, see
Section~\ref{sec:left.cov.calculus}). Namely, defining $X_a$ for
$a=1,2,3$ as
\begin{align*}
  X_1 \equiv \Xp = \sqrt{q}EK\qquad
  X_2 \equiv \Xm  = \frac{1}{\sqrt{q}}FK\qquad
  X_3 \equiv \Xz = \frac{1-K^4}{1-q^{-2}}
\end{align*}
it follows that for $f,g\in\Sthreeq$ (where $X_a(f)$ denotes either $X_a\rtr f$ or $f\ltr X_a$)
\begin{align*}
  &\Xp(fg) = f\Xp(g) + \Xp(f)\sigmap(g)\\
  &\Xm(fg) = f\Xm(g) + \Xp(f)\sigmam(g)\\
  &\Xz(fg) = f\Xz(g) + \Xp(f)\sigmaz(g),
\end{align*}
with
\begin{align}\label{sig}
\sigmap=\sigmam=K^2\quad\text{and}\quad \sigmaz=K^4. 
\end{align} 
Furthermore, these maps satisfy the following $q$-deformed commutation relations
\begin{align}
  &\Xm\Xp - q^2\Xp\Xm = \Xz\label{eq:Xmp.com}\\
  q^2&\Xz\Xm-q^{-2}\Xm\Xz=(1+q^2)\Xm\label{eq:Xzm.com}\\
  q^2&\Xp\Xz-q^{-2}\Xz\Xp=(1+q^2)\Xp.\label{eq:Xzp.com}
\end{align}
For an arbitrary map $X:\Stq\to\Stq$ one defines $X^\ast:\Stq\to\Stq$ as
\begin{align*}
  X^\ast(f) = \paraa{X(f^\ast)}^\ast
\end{align*}
and it follows that
\begin{align}\label{X.Xs.relation}
  \Xp^\ast = -K^{-2}\Xm\qquad
  \Xm^\ast = -K^{-2}\Xp\qquad
  \Xz^\ast = -K^{-4}\Xz,
\end{align}
satisfying
\begin{align}
  &\Xp^\ast(fg) = \sigmasp(f)\Xp^\ast(g)+\Xp^\ast(f)g  \label{11} \\
  &\Xm^\ast(fg) = \sigmasm(f)\Xm^\ast(g)+\Xm^\ast(f)g \label{12} \\
  &\Xz^\ast(fg) = \sigmasz(f)\Xz^\ast(g)+\Xz^\ast(f)g  \label{13}
\end{align}
with 
\begin{align}\label{sigstar}
\sigmasp=\sigmap^{-1}=K^{-2}, \qquad \sigmasm=\sigmam^{-1}=K^{-2}, \qquad
\sigmasz=\sigmaz^{-1}=K^{-4}.
\end{align}
We stress that $X^*_a$ is different from $X^\dag_a$, as defined above on $\Uqsu$.

\subsection{A left covariant calculus on $\Stq$}\label{sec:left.cov.calculus}

\noindent
It is well known that there is a left covariant (first order)
differential calculus on $\Sthreeq$, denoted by $\OmegaStq$, generated
as a left $\Stq$-module by
\begin{align*}
  \omega_1=\omegap = a\,dc-qc\,da\qquad
  \omega_2=\omegam = c^\ast da^\ast-qa^\ast dc^\ast\qquad
  \omega_3=\omegaz = a^\ast da + c^\ast dc,
\end{align*}
together with the differential $d:\Sthreeq\to\OmegaStq$
\begin{align}\label{df3}
  df = (\Xp\rtr f)\omegap + (\Xm\rtr f)\omegam + (\Xz\rtr f)\omegaz
\end{align}
for $f\in\Sthreeq$ \cite{w:twisted.su2}. In fact, $\OmegaStq$ is a free left module with a
basis given by $\{\omegap,\omegam,\omegaz\}$. Moreover, $\OmegaStq$ is
a bimodule with respect to the relations
\begin{alignat*}{4}
  &\omega_{z}a = q^{-2}a\omega_{z} &\qquad
  &\omega_{z}a^* = q^{2}a^*\omega_{z} &\qquad 
  &\omega_{z}c = q^{-2}c\omega_{z} &\qquad
  &\omega_{z}c^* = q^{2}c^*\omega_{z} \\
  &\omega_{\pm} a = q^{-1}a\omega_{\pm} &\qquad
  &\omega_{\pm}a^* = q a^*\omega_{\pm}&\qquad
  &\omega_{\pm}c = q^{-1}c\omega_{\pm} &\qquad 
  &\omega_{\pm}c^* = qc^*\omega_{\pm},
\end{alignat*}
and, furthermore, $\OmegaStq$ is a $\ast$-bimodule with
\begin{align*}
  &\omegap^\dag= -\omegam\qquad\omegaz^\dag=-\omegaz
\end{align*}
satisfying $(f\omega g)^\dagger=g^\ast\omega^\dag f^\ast$ for
$f,g\in\Sthreeq$ and $\omega\in\OmegaStq$.

\section{$q$-affine connections}\label{sec:q.affine.connections}

\noindent
In differential geometry, a connection extends the action of
derivatives to vector fields, and for $\Stq$ a natural set of
($q$-deformed) derivations is given by
$\{X_a\}_{a=1}^3=\{\Xp,\Xm,\Xz\}$. In this section, we will introduce
a framework extending the action of $X_a$ to a connection on
$\Stq$-modules. Let us first define the set of derivations we shall be
interested in.

\begin{definition}
  The quantum tangent space of $\Stq$ is defined as
  \begin{align*}
    \TStq = \complex\angles{\Xp,\Xp^\ast,\Xm,\Xm^\ast,\Xz,\Xz^\ast},
  \end{align*}
that is the complex vector space generated by $X_a$ and $X_a^\ast$ for
  $a=1,2,3$.
\end{definition}

\noindent
Considering $\TStq$ to be the analogue of a (complexified) tangent space of $\Stq$,
we would like to introduce a covariant derivative $\nabla_X$ on a (left)
$\Stq$-module $M$, for $X\in\TStq$. Since the basis elements of
$\TStq$ act as $q$-deformed derivations, the connection should obey an
analogous $q$-deformed Leibniz rule.  The motivating example is when
$M=\Stq$ and the action of $\TStq$ is simply $\nabla_{X}f = X(f)$ for
$X\in\TStq$ and $f\in\Stq$. In fact, let us be slightly more general
and consider the action on a free module of rank $n$. Thus, we let $M$
be a free left $\Stq$-module with basis $\{e_i\}_{i=1}^n$, and
  write an arbitrary element $m\in M$ as $m=m^ie_i$ for
  $m^i\in\Stq$, implicitly assuming a summation over $i$ from $1$ to $n$.
Moreover, we assume 
there exist $\complex$-linear maps $\sigmah_a,\sigmahs_a:M\to M$ such that
\begin{align*}
  \sigmah_a(fm)=\sigma_a(f)\sigmah_a(m) \qand 
  \sigmah^\ast_a(fm)=\sigmas_a(f)\sigmah^\ast_a(m)
\end{align*}
for $f\in\Stq$, $m\in M$ and $a=1,2,3$ (for instance, one may choose
$\sigmah_a(m^ie_i)=\sigma_a(m^i)e_i$ and similarly for $\sigmahs_a$).
Let us define $\nzero:\TStq\times M\to M$ by setting
\begin{align}\label{eq:nzero.def}
  \nzero_{X_a}(m) = X_a(m^i)\sigmah_a(e_i) \qand
  \nzero_{X^\ast_a}(m) = X^\ast_a(m^i)e_i
\end{align}
for $m=m^ie_i\in M$ (and extending it linearly to all of $\TStq$).  Now, it is easy to check that
\begin{align}
  &\nzero_{X_a}(fm) = f\nzero_{X_a}m + X_a(f)\sigmah_a(m)\label{eq:nzero.X.leibniz}\\
  &\nzero_{X_a^\ast}(fm) = \sigmas_a(f)\nzero_{X^\ast_a}m + X^\ast_a(f)m \label{eq:nzero.Xs.leibniz}
\end{align}
for $f\in\Stq$ and $m\in M$.  Let us generalize these concepts to
arbitrary $\Stq$-modules supporting an action of $\sigma_a$.
\begin{definition}\label{def:sigma.module}
  Let $M$ be a left $\Stq$-module and let
  $\sigmah_a,\sigmah^\ast_a:M\to M$ be maps such that
  \begin{align*}
    &\sigmah_a(\lambda_1m_1+\lambda_2m_2) =
      \lambda_1\sigmah_a(m_1)+\lambda_2\sigmah_a(m_2)\qquad
    \sigmah_a(fm) = \sigma_a(f)\sigmah_a(m)\\
    &\sigmah^\ast_a(\lambda_1m_1+\lambda_2m_2) =
      \lambda_1\sigmah^\ast_a(m_1)+\lambda_2\sigmah^\ast_a(m_2)\qquad
    \sigmah^\ast_a(fm) = \sigma^\ast_a(f)\sigmah^\ast_a(m)
  \end{align*}
  for $\lambda_1,\lambda_2\in\complex$, $f\in\Stq$, 
  $m_1,m_2,m\in M$ and $a=1,2,3$. Then $(M,\sigmah_a,\sigmah^\ast_a)$ is called a
  \emph{$\sigma$-module}. Moreover, given the $\sigma$-modules
  $(M,\sigmah_a,\sigmahs_a)$ and $(\tilde{M},\tilde{\sigma}_a,\tilde{\sigma}_a^\ast)$, a left
  module homomorphism $\phi:M\to \tilde{M}$ is called a \emph{$\sigma$-module
    homomorphism} if
  \begin{align*}
    \phi\paraa{\sigmah_a(m)} = \tilde{\sigma}_a\paraa{\phi(m)} z \qand 
    \phi\paraa{\sigmahs_a(m)} = \tilde{\sigma}_a^\ast\paraa{\phi(m)}
  \end{align*}
  for $m\in M$ and $a=1,2,3$.
  
\end{definition}

\noindent
For notational convenience, when the maps $\sigmah_a,\sigmahs_a$ are
clear from the context, we shall simply write $M$ for the
$\sigma$-module $(M,\sigmah,\sigmahs)$. Next, motivated by
\eqref{eq:nzero.X.leibniz} and \eqref{eq:nzero.Xs.leibniz}, we
introduce connections on $\sigma$-modules.

\begin{definition}\label{def:q.affine.connection}
  Let $M$ be a left $\sigma$-module. A \emph{left $q$-affine connection on
    $M$} is a map $\nabla:\TStq\times M\to M $ such that
  \begin{enumerate}
  \item $\nabla_{X}(\lambda_1m_1+\lambda_2m_2) = \lambda_1\nabla_Xm_1 + \lambda_2\nabla_Xm_2$,\label{q.affine.lin.module}
  \item $\nabla_{\lambda_1 X+\lambda_2Y}m = \lambda_1\nabla_Xm + \lambda_2\nabla_Ym$,\label{q.affine.lin.tangent}
  \item $\nabla_{X_a}(fm) = f\nabla_{X_a}m+X_a(f)\sigmah_a(m)$,\label{q.affine.Xa.leibniz}
  \item $\nabla_{X^\ast_a}(fm) = \sigmas_a(f)\nabla_{X^\ast_a}m+X^\ast_a(f)m$,\label{q.affine.Xas.leibniz}
  \end{enumerate}
  for $a=1,2,3$, $m,m_1,m_2\in M$, $f\in\Stq$, $X\in\TStq$ and
  $\lambda_1,\lambda_2\in\complex$.
\end{definition}

\begin{remark}
  Note that given $\nabla_{X_a}$ for $a=1,2,3$, one can set
  \begin{align*}
    \nabla_{\Xp^\ast}=-\sigmahsm\circ\nabla_{\Xm}\qquad
    \nabla_{\Xm^\ast}=-\sigmahsp\circ\nabla_{\Xp}\qquad
    \nabla_{\Xz^\ast}=-\sigmahsz\circ\nabla_{\Xz}
  \end{align*}
  satisfying (\ref{q.affine.Xas.leibniz}) in
    Definition~\ref{def:q.affine.connection}, due to \eqref{X.Xs.relation} and \eqref{sigstar}.
\end{remark}

\noindent
Next, assume that the module $M$ 
comes with a hermitian form $h:M\times M\to\Stq$ satisfying
\begin{align*}
  &h(fm_1,m_2)=fh(m_1,m_2)\qquad  h(m_1,m_2)^\ast = h(m_2,m_1)\\
  &h(m_1+m_2,m_3) = h(m_1,m_3)+h(m_2,m_3)
\end{align*}
for $f\in\Stq$ and $m_1,m_2,m_3\in M$. On a free module with basis
$\{e_i\}_{i=1}^n$, a hermitian form is given by $h_{ij}=h_{ji}^\ast\in \Stq$
by setting
\begin{align}\label{eq:canon.metric.free}
  h(m_1,m_2) = m_1^ih_{ij}(m_2^j)^\ast
\end{align}
for $m_1=m_1^ie_i\in(\Stq)^n$ and $m_2=m_2^ie_i\in(\Stq)^n$. In the
case of the $q$-affine connection $\nzero$ in \eqref{eq:nzero.def}, one finds that
\begin{align*}
  \Xp\paraa{&h(m_1,m_2)}
              = \Xp\paraa{m_1^ih_{ij}(m_2^j)^\ast}
              = m_1^i\Xp\paraa{h_{ij}(m_2^j)^\ast} +
              \Xp(m_1^i)\sigmap\paraa{h_{ij}(m_2^j)^\ast}\\
            &=m_1^ih_{ij}\Xp\paraa{(m_2^j)^\ast}+
              m_1^i\Xp(h_{ij})\sigmap\paraa{(m_2^j)^\ast}
              +\Xp(m_1^i)\sigmap\paraa{h_{ij}(m_2^j)^\ast},
\end{align*}
and assuming that $\Xp(h_{ij})=0$ one obtains
\begin{align*}
  \Xp\paraa{h(m_1,m_2)}
  &= m_1^ih_{ij}\paraa{\Xp^\ast(m_2^j)}^\ast
    + \sigmap\paraa{(\sigmap^{-1}\circ\Xp)(m_1^i)h_{ij}(m_2^j)^\ast}\\
  &= h\paraa{m_1,\nzero_{\Xp^\ast}(m_2)}
    -\sigmap\paraa{h(\nzero_{\Xm^\ast} m_1,m_2)},
\end{align*}
by using that $\Xm^\ast=-\sigmap^{-1}\circ\Xp$.  Corresponding
formulas are easily worked out for $\nzero_{\Xm},\nzero_{\Xz}$, and we
shall take this as a motivation for the following definition.
\begin{definition}\label{def:metric.compatibility}
  A $q$-affine connection $\nabla$ on a left $\sigma$-module $M$ is
  compatible with the hermitian form $h:M\times M\to\Stq$ if
  \begin{align}
    &\Xp\paraa{h(m_1,m_2)}
      = -\sigmap\paraa{h(\nabla_{\Xm^\ast} m_1,m_2)}
      +h\paraa{m_1,\nabla_{\Xp^\ast} m_2}\label{eq:metric.comp.p}\\
    &\Xm\paraa{h(m_1,m_2)}
      = -\sigmam\paraa{h(\nabla_{\Xp^\ast} m_1,m_2)}
      +h\paraa{m_1,\nabla_{\Xm^\ast} m_2}\label{eq:metric.comp.m}\\
    &\Xz\paraa{h(m_1,m_2)}
      =-\sigmaz\paraa{h(\nabla_{\Xz^\ast} m_1,m_2)}
      +h\paraa{m_1,\nabla_{\Xz^\ast} m_2},\label{eq:metric.comp.z}
  \end{align}
  for $m_1,m_2\in M$.
\end{definition}

\noindent
Note that \eqref{eq:metric.comp.p} and \eqref{eq:metric.comp.m} are equivalent since
\begin{align*}
  \parab{
  \Xp&\paraa{h(m_2,m_1)}
      + \sigmap\paraa{h(\nabla_{\Xm^\ast} m_2,m_1)}
      -h\paraa{m_2,\nabla_{\Xp^\ast} m_1}
  }^\ast\\
     &=-K^{-2}\parab{
       \Xm\paraa{h(m_1,m_2)}
      + \sigmam\paraa{h(\nabla_{\Xp^\ast} m_1,m_2)}
      -h\paraa{m_1,\nabla_{\Xm^\ast} m_2}
       }.
\end{align*}
In the case of a $q$-affine connection on a free module, one can
derive a convenient parametrization of all connections that are
compatible with a given hermitian form. To this end, let us introduce
some notation. Let $(\Stq)^n$ be a free $\sigma$-module with basis
$\{e_i\}_{i=1}^n$.  A $q$-affine connection $\nabla$ on $(\Stq)^n$ can
be determined by specifying the Christoffel symbols
\begin{align*}
  \nabla_{X_a}e_i = \Gamma_{ai}^je_j,
\end{align*}
with $\Gamma_{ai}^j\in\Stq$ for $a=1,2,3$ and $i,j=1,\ldots,n$, and
setting
\begin{align*}
  &\nabla_{\Xp^\ast}e_i = -\sigmah_-^\ast\paraa{\nabla_{\Xm}e_i}\qquad
  \nabla_{\Xm^\ast}e_i = -\sigmah_+^\ast\paraa{\nabla_{\Xp}e_i}\qquad
  \nabla_{\Xz^\ast}e_i = -\sigmah_z^\ast\paraa{\nabla_{\Xz}e_i}\\
  &\nabla_{X_a}(m^ie_i) = m^i\nabla_{X_a}e_i+X_a(m^i)\sigmah_a(e_i)\qquad
  \nabla_{X^\ast_a}(m^ie_i) = \sigma_a^\ast(m^i)\nabla_{X_a^\ast}e_i+X^\ast_a(m^i)e_i.
\end{align*}

\noindent 
As we shall see, the metric compatibility of
$\nabla$ is conveniently formulated in terms of
\begin{align}
  \Gammat_{ai,j} = \Gamma_{ai}^k\sigma_a(\htilde_{akj})\qquad 
  \htilde_{aij}=h\paraa{\sigmahs_a(e_i),e_j}.\label{eq:Gammat.def}
\end{align}
The hermitian form $h$ is assumed to be invertible (i.e. inducing an
isomorphism of the module and its dual) which implies that there
exists $h^{ij}$ such that
$h_{ij}h^{jk}=h^{kj}h_{ji}=\delta_i^k\mid$. In this case, one finds
that $\htilde_{aij}$ (for $a=1,2,3$) is invertible as well,
implying that one may invert \eqref{eq:Gammat.def} as
\begin{align*}
  \Gamma_{aj}^i=\Gammat_{aj,k}\sigma_a\paraa{\htilde_a^{ki}}.
\end{align*}

\begin{proposition}\label{prop:nabla.comp.metric}
  Let $(\Stq)^n$ be a free $\sigma$-module with a basis
  $\{e_i\}_{i=1}^n$ and let $\nabla$ be a $q$-affine connection on
  $(\Stq)^n$ given by the Christoffel symbols
  $\nabla_ae_i = \Gamma_{ai}^je_j$. Furthermore, assume that $h$ is
  an invertible hermitian form on $(\Stq)^n$. Then $\nabla$ is
  compatible with $h$ if and only if there exist hermitian matrices
  $\alpha,\beta,\rho\in\Mat_n(\Stq)$ such that
  \begin{align}
    &\Gammat_{+i,j} = \tfrac{1}{2}\Xp(h_{ij})+K(\alpha_{ij})+iK(\beta_{ij})\\
    &\Gammat_{-i,j} = \tfrac{1}{2}\Xm(h_{ij})+K(\alpha_{ij})-iK(\beta_{ij})\\
    &\Gammat_{zi,j} = \tfrac{1}{2}\Xz(h_{ij})+K^2(\rho_{ij}),
  \end{align}
  with $h_{ij}=h(e_i,e_j)$.
\end{proposition}

\begin{proof}
  Starting from $\nabla_{X_a}e_i=\Gamma_{ai}^je_j$ one obtains
  \begin{align*}
    -\sigmap \paraa{h(\nabla_{\Xm^\ast}e_i,e_j)}+h(e_i,\nabla_{\Xp^\ast}e_j)
            &=\sigmap\paraa{h(\sigmahsp(\Gamma_{+i}^ke_k),e_j)}
              -h\paraa{e_i,\sigmahsm(\Gamma_{-j}^ke_k)}\\
            &=\Gamma^k_{+i}\sigmap\paraa{h(\sigmahsp(e_k),e_j)}
              -h(e_i,\sigmahsm(e_k))\sigmasm(\Gamma_{-j}^k)^\ast\\
            &=\Gammat_{+i,j}-\sigmasm\paraa{\Gammat_{-j,i}}^\ast
  \end{align*}
  giving the metric compatibility equation \eqref{eq:metric.comp.p} as
  \begin{align}\label{eq:metric.p.gammat}
    \Gammat_{+i,j} = \Xp\paraa{h_{ij}} + K^{-2}\paraa{\Gammat_{-j,i}}^\ast.
  \end{align}
  Similarly, equation \eqref{eq:metric.comp.z} may be written as
  \begin{align}\label{eq:metric.z.gammah}
    \Gammat_{zi,j} = \Xz\paraa{h_{ij}} + K^{-4}\paraa{\Gammat_{zj,i}}^\ast.
  \end{align}
  To solve \eqref{eq:metric.p.gammat} one may freely choose
  $\Gammat_{-i,j}$ and define $\Gammat_{+i,j}$ accordingly. Without
  loss of generality, let us write $\Gammat_{-i,j}$ in the following form
  \begin{align*}
    \Gammat_{-i,j} = \tfrac{1}{2}\Xm(h_{ij}) + K(\alpha_{ij}) - iK(\beta_{ij})
  \end{align*}
  for arbitrary $\alpha_{ij}=\alpha_{ji}^\ast$ and
  $\beta_{ij}=\beta_{ji}^\ast$. Then one obtains
  \begin{align*}
    \Gammat_{+i,j} &= \Xp\paraa{h_{ij}} + K^{-2}\paraa{\Gammat_{-j,i}}^\ast\\
                   &= \Xp\paraa{h_{ij}} + \tfrac{1}{2}K^2\paraa{\Xm(h_{ji})^\ast}
                     + K^2\paraa{K(\alpha_{ji})^\ast}+iK^2\paraa{K(\beta_{ji})^\ast}\\
                   &=\Xp\paraa{h_{ij}} - \tfrac{1}{2}K^2\paraa{K^{-2}\Xp(h_{ij})}
                     +K^2\paraa{K^{-1}(\alpha_{ij})}+iK^2\paraa{K^{-1}(\beta_{ij})}\\
                   &=\tfrac{1}{2}\Xp(h_{ij}) + K(\alpha_{ij})+iK(\beta_{ij})
  \end{align*}
  Hence, every solution of \eqref{eq:metric.comp.p} may be written in
  the above form.  For the metric equation \eqref{eq:metric.comp.z}
  one writes $\Gammat_{zi,j} = K^2(\rho_{ij}+ig_{ij})$, with
  $\rho_{ij}=\rho_{ji}^\ast$ and $g_{ij}=g_{ji}^\ast$, and notes that
  \eqref{eq:metric.z.gammah} is equivalent to
  \begin{align*}
    K^2(\rho_{ij}+ig_{ij}) = \Xz(h_{ij})+K^2(\rho_{ij}-ig_{ij})\equivalent
    iK^2(g_{ij}) = \tfrac{1}{2}\Xz(h_{ij}).
  \end{align*}
  This is compatible with the requirement that $g_{ij}=g_{ji}^\ast$ since
  \begin{align*}
    \paraa{K^{-2}\Xz(h_{ji})}^\ast = K^2\paraa{\Xz(h_{ji})^\ast}=-K^2\paraa{K^{-4}\Xz(h_{ji})}
    =-K^{-2}\Xz(h_{ij}).
  \end{align*}
  Hence, the general solution to \eqref{eq:metric.z.gammah} can be written as
  \begin{align*}
    \Gammat_{zi,j} = \tfrac{1}{2}\Xz(h_{ij}) + K^2(\rho_{ij})
  \end{align*}
  for arbitrary $\rho_{ij}=\rho_{ji}^\ast$, which concludes the proof.
\end{proof}

\subsection{$q$-affine connections on projective modules}

\noindent
As expected, $q$-affine connections exist on projective
modules. More precisely, one proves the following result.

\begin{proposition}\label{prop:q.affine.projective}
  Let $M=\paraa{(\Stq)^n,\sigmah^0_a,(\sigma_a^0)^\ast}$ be a free
  $\sigma$-module and let $\nzero$ be a $q$-affine connection on
  $M$. If $p:(\Stq)^n\to(\Stq)^n$ is a projection  
  and $\sigmah_a=p\circ\sigmah^0_a$, 
  $\sigmahs_a=p\circ(\sigmah_a^0)^\ast$, then
  \begin{align*}
    (p(\Stq)^n,\sigmah_a,\sigmahs_a),
  \end{align*}
  is a $\sigma$-module and
  $p\circ\nzero$ is a $q$-affine connection on $(p(\Stq)^n,\sigmah_a,\sigmahs_a)$.
\end{proposition}

\begin{proof}
  It follows immediately that
  $(p(\Stq)^n,p\circ\sigmah^0_a,p\circ(\sigmah_a^0)^\ast)$ satisfy the
  requirements of Definition~\ref{def:sigma.module}, since
  $\paraa{(\Stq)^n,\sigmah^0_a,(\sigma_a^0)^\ast}$ is a
  $\sigma$-module. For instance,
  \begin{align*}
    \sigmah_a(fm) = p\paraa{\sigmah_a^0(fm)}
    = p\paraa{\sigma_a(f)\sigmah_a^0(m)}
    =\sigma_a(f)p\paraa{\sigmah_a^0(m)}
    =\sigma_a(f)\sigmah_a(m).
  \end{align*}
  Since $\nzero$ is a $q$-affine connection, it is immediate that
  $\nabla=p\circ\nzero$ satisfies properties
  (\ref{q.affine.lin.module}) and (\ref{q.affine.lin.tangent}) in
  Definition~\ref{def:q.affine.connection}. Moreover, for $m\in p(\Stq)^n$
  \begin{align*}
    \nabla_{X_a}(fm)
    &=p\nzero_{X_a}(fm) = fp\paraa{\nzero_{X_a}m} + X_a(f)(p\circ\sigmah^0_a)(m)\\
    &=f\nabla_{X_a}m + X_a(f)\sigmah_a(m)\\
    \nabla_{X^\ast_a}(fm)
    &=\sigma_a^\ast(f)p\paraa{\nzero_{X^\ast_a}m} + X^\ast_a(f)p(m)\\
    &=\sigma_a^\ast(f)\nabla_{X^\ast_a}m + X_a^\ast(f)m,
  \end{align*}
  from which we conclude that $\nabla$ is a $q$-affine connection on
  $p(\Stq)^n$.
\end{proof}

\noindent
Since we have shown in the previous section that one can construct
$q$-affine connections on free modules,
Proposition~\ref{prop:q.affine.projective} shows that $q$-affine
connections exist on projective modules.  Moreover, Let $\nabla$ and $\nablat$
be $q$-affine connections on a $\sigma$-module $M$ and define
\begin{align*}
  \alpha(X,m) = \nabla_Xm-\nablat_Xm.
\end{align*}
Then $\alpha:\TStq\times M\to M$ satisfies
\begin{align}
  &\alpha(\lambda X+Y,m_1) = \lambda\alpha(X,m_1)+\alpha(Y,m_1)\label{eq:alpha.prop.1}\\
  &\alpha(X,fm_1+m_2) = f\alpha(X,m_1)+\alpha(X,m_2)\label{eq:alpha.prop.2}
\end{align}
for $m_1,m_2\in M$, $X\in\TStq$, $f\in\Stq$ and
$\lambda\in\complex$. Conversely, every $q$-affine connection on a
projective module $M$ can be written as
\begin{align*}
  \nabla_Xm = p(\nzero_X m) + \alpha(X,m).
\end{align*}
where $\nzero$ is the connection defined in \eqref{eq:nzero.def} and
$\alpha:\TStq\times M\to M$ is an arbitrary map satisfying
\eqref{eq:alpha.prop.1} and \eqref{eq:alpha.prop.2}.
Next, let us show that a connection on a projective module is
  compatible with the restricted metric if the projection is
  orthogonal.

\begin{proposition}
  Let $\nabla$ be a $q$-affine connection on the free $\sigma$-module
  $(\Stq)^n$ and assume furthermore that $\nabla$ is compatible with a
  hermitian form $h$ on $(\Stq)^n$. If $p:(\Stq)^n\to(\Stq)^n$ is an
  orthogonal projection, i.e.
  \begin{align*}
    h\paraa{p(m_1),m_2} = h\paraa{m_1,p(m_2)}
  \end{align*}
  for all $m_1,m_2\in(\Stq)^n$, then $\nablat=p\circ\nabla$ is compatible with
  $h$ restricted to $p(\Stq)^n$.
\end{proposition}

\begin{proof}
  Let us explicitly check one of the conditions in
  Definition~\ref{def:metric.compatibility} for $m_1,m_2\in p(\Stq)^n$:
  \begin{align*}
    -\sigma_+\paraa{h(&\nablat_{\Xm^\ast}m_1,m_2)}+h\paraa{m_1,\nablat_{\Xp^\ast}m_2}
    =-\sigma_+\paraa{h(p\nabla_{\Xm^\ast}m_1,m_2)}+h\paraa{m_1,p\nabla_{\Xp^\ast}m_2}\\
    &= -\sigma_+\paraa{h(\nabla_{\Xm^\ast}m_1,p(m_2))}+h\paraa{p(m_1),\nabla_{\Xp^\ast}m_2}\\
    &= -\sigma_+\paraa{h(\nabla_{\Xm^\ast}m_1,m_2)}+h\paraa{m_1,\nabla_{\Xp^\ast}m_2}=\Xp\paraa{h(m_1,m_2)}.
  \end{align*}
  The remaining conditions are checked in an analogous way.
\end{proof}

\section{A $q$-affine Levi-Civita connection on $\OmegaStq$} \label{se:4}

\noindent
In this section we shall construct a $q$-affine connection on the free
left module $\OmegaStq$, compatible with a hermitian form $h$, 
satisfying a certain torsion freeness condition. The module
$\OmegaStq$ is a free $\Stq$-module of rank 3 with basis
$\omegap,\omegam,\omegaz$ which implies that the results of
Proposition~\ref{prop:nabla.comp.metric} may be used.  To start with,
one has to endow $\OmegaStq$ with the structure of a $\sigma$-module.
Firstly, the actions \eqref{actions} are extended to forms by requiring they commute 
with the differential $d$. Then, one checks directly that    
$K \rtr \omegap = q^{-1}\omegap$, $K \rtr \omegam = q \omegam$, 
$K \rtr \omegaz = \omegaz$, while 
$K \ltr \omega_a = \omega_a$, for $a=1,2,3$. 
Let us work with the right action (without indicating it explicitly). 
The $\sigma$-module structure is then introduced as
\begin{align*}
  K(\omega_a)=\omega_a\implies
  \sigmah_a(m^b\omega_b) = \sigma_a(m^b)\omega_b\qquad
  \sigmahs_a(m^b\omega_b) = \sigmas_a(m^b)\omega_b  .
\end{align*}
Furthermore, we assume that $h$ is an invertible hermitian form on $\OmegaStq$, such that 
$K(h_{ab})=h_{ab}$, with $h_{ab} = h(\omega_a,\omega_b)$, for $a,b=1,2,3$. 
(Note that for such a metric one has
$X_z(h_{ab})=0$.)  With these choices, one finds that
$\Gammat_{ab,c}= \Gamma_{ab}^ph_{pc}$ where the Christoffel
  symbols are defined as
  $\nabla_{X_a}\omega_b = \Gamma_{ab}^c\omega_c$ for $a,b=1,2,3$.
In the case of a $q$-affine connection on $\OmegaStq$, there is a
natural definition of torsion, motivated by
\eqref{eq:Xmp.com}--\eqref{eq:Xzp.com}.

\begin{definition}
  A $q$-affine connection $\nabla$ on $\OmegaStq$ is \emph{torsion free} if
  \begin{align}
    &\nablam\omegap-q^2\nablap\omegam=\omegaz\\
    q^2&\nablaz\omegam-q^{-2}\nablam\omegaz=(1+q^2)\omegam\\
    q^2&\nablap\omegaz-q^{-2}\nablaz\omegap=(1+q^2)\omegap.
  \end{align}
\end{definition}

\noindent
In terms of $\Gammat_{ab,c}$, the conditions
for a torsion free connection may be reformulated as
\begin{align}
  &\Gammat_{-+,a}-q^2\Gammat_{+-,a}=K^2\paraa{\htilde_{+za}}\\
  q^2&\Gammat_{z-,a}-q^{-2}\Gammat_{-z,b}\sigmam(\htilde_{-}^{bc})\sigmaz(\htilde_{zca})
    =(1+q^2)K^4(\htilde_{z-a})\\
  q^{-2}&\Gammat_{z+,a}-q^{2}\Gammat_{+z,b}\sigmap(\htilde_{+}^{bc})\sigmaz(\htilde_{zca})
    =-(1+q^2)K^4(\htilde_{z+a})
\end{align}
for $a=1,2,3$ and an arbitrary $\sigma$-module structure on $\OmegaStq$. 
For the particular case when the metric is invariant by $K$, that is $K(h_{ab})=h_{ab}$, and
$K(\omega_a)=\omega_a$ for $a,b=1,2,3$, the torsion free equations become
\begin{align}
  &\Gammat_{-+,a}-q^2\Gammat_{+-,a} = h_{za}  \label{ga1} \\
  q^2&\Gammat_{z-,a}-q^{-2}\Gammat_{-z,a} = (1+q^2)h_{-a} \label{ga2}  \\
  q^{2}&\Gammat_{+z,a}-q^{-2}\Gammat_{z+,a} = (1+q^2)h_{+a}. \label{ga3} 
\end{align}
Out of these, in Section~\ref{sec:appendix}, we derive an explicit
expression for a torsion free and metric connection on $\OmegaStq$,
and show that such a connection exists if the following reality condition is satisfied
(cf. eq. \eqref{eq:nec.lc.metric.cond.real}):
\begin{align}\label{eq:cond.exist.connection}
  \paraa{\Xp(h_{-z}-q^{-4}h_{z+})}^\ast = \Xp(h_{-z}-q^{-4}h_{z+}).
\end{align}
The connection is not unique, and the solution depends on 6 parameters.
For the particular case when $h_{ab}=h\delta_{ab}$ (obviously
satisfying \eqref{eq:cond.exist.connection}), one can set all parameters to be zero:
$ \tau_1 = \tau_4 = \gamma_{+-} = f_0 = \mu_2 =\rho_{zz} = 0$, in the notation of
Section~\ref{sec:appendix}. It follows that a
torsion free and metric connection is given by
\begin{align*}
  &\nablap\omegap = \Xp(h)h^{-1}\omegap\\
  &\nablap\omegam = -\thalf q^{-2}\omegaz\\
  &\nablap\omegaz = \thalf\omegap +\thalf\Xp(h)h^{-1}\omegaz\\
  &\nablam\omegap  = \thalf\omegaz\\
  &\nablam\omegam = \Xm(h)h^{-1}\omegam\\
  &\nablam\omegaz = -\thalf q^{-2}\omegam+\thalf\Xm(h)h^{-1}\omegaz\\
  &\nablaz\omegap = -\thalf q^2(2+q^2)\omegap+\thalf q^{4}\Xp(h)h^{-1}\omegaz\\
  &\nablaz\omegam = \thalf(2+2q^{-2}-q^{-6})\omegam
    + \thalf q^{-4}\Xm(h)h^{-1}\omegaz\\
  &\nablaz\omegaz = -\thalf q^2\Xm(h)h^{-1}\omegap
    -\thalf q^{-2}\Xp(h)h^{-1}\omegam \, .
\end{align*}

\section{The quantum 2-sphere}

\noindent
The noncommutative (standard) Podle\'s sphere $S^2_q$
\cite{p:quantum.spheres} can be considered as a subalgebra of $\Stq$
by identifying the generators $\Bz,\Bp,\Bm$ of $\Stwoq$ as
\begin{align*}
  \Bz = cc^\ast\qquad
  \Bp = ca^\ast\qquad
  \Bm = ac^\ast = \Bp^\ast ,
\end{align*}
satisfying then the relations
\begin{alignat*}{2}
\Bm\,\Bz &= q^{2}\, \Bz\,\Bm   &\qquad &\Bp\,\Bz = q^{-2}\, \Bz\,\Bp    \\ 
\Bm\,\Bp &= q^2\, \Bz \,\big( \mid - q^2\, \Bz \big)  &\qquad &\Bp\,\Bm= \Bz \,\big( \mid - \Bz \big) . 
\end{alignat*}
These elements generate the fix-point algebra of the right $U(1)$-action
\begin{align*}
  \alpha_z(a) =az\qquad\alpha_z(a^\ast)=a^\ast\bar{z}\qquad
  \alpha_z(c) =cz\qquad\alpha_z(c^\ast)=c^\ast\bar{z}
\end{align*}
for $z\in U(1)$ and $a\in\Stq$, related to the $U(1)$-Hopf-fibration $\Stwoq\hookrightarrow \Stq$. 

Now, the left action of the $X_a$ does not preserve the algebra $\Stwoq$: one readily computes,
\begin{alignat*}{3}
  &\Xp \rtr \Bz = q a^\ast c^\ast &\quad
  &\Xm \rtr \Bz = -q^{-1} ca &\quad
  &\Xz \rtr \Bz = 0\\
  &\Xp \rtr \Bp = q (a^\ast)^2 &\quad
  &\Xm \rtr \Bp = c^2 &\quad
  &\Xz \rtr \Bp = 0\\
  &\Xp \rtr \Bm = q^{2} (c^\ast)^2 &\quad
  &\Xm \rtr \Bm = -q^{-1} (a)^2 &\quad
  &\Xz \rtr \Bm = 0.
\end{alignat*}
On the other hand, the right action of $X_a$ does preserve the algebra $\Stwoq$. 
Let us denote $Y_a=X_a$ for the right action. Then, it is easy to check that
\begin{alignat*}{3}
  &\Bz\ltr\Yp = q^{-1}\Bm &\quad
  &\Bz\ltr\Ym = -q^{-1}\Bp &\quad
  &\Bz\ltr\Yz = 0\\
  &\Bp\ltr\Yp = q\mid -q(1+q^2)\Bz &
  &\Bp\ltr\Ym = 0 &
  &\Bp\ltr\Yz = -q^2(1+q^2)\Bp\\
  &\Bm\ltr\Yp = 0 &
  &\hspace{-15mm}\Bm\ltr\Ym = -q^{-1}\mid + q^{-1}(1+q^2)\Bz &
  &\Bm\ltr\Yz = (1+q^{-2})\Bm.
\end{alignat*}
Note that when restricted to $\Stwoq$ the $Y_a$ are not independent. 
A long but straightforward computation shows that they are indeed related as
\begin{equation}\label{rel-rvf}
  \begin{split}
    \big( (f \ltr \Yp) \Bp \, q & + (f \ltr \Ym)  \Bm \, q^{-1} \big)(1+q^2) +
    (f \ltr \Yz) \para{ 1-2 \frac{1+q^2}{1+q^4} \Bz}\\
    & = (f \ltr \Yz^2) \, q^{-2} \para{ \frac{1-q^2}{1+q^4} \, ( 2 q^4 + q^2 + 1 ) \Bz - (1-q^6) \Bz^2} \\
    &\quad+ (f \ltr K^4) \, q^{-2} (1+q^2) \paraa{ (q^4 -1) \Bz + (1-q^6) \Bz^2},
  \end{split}
\end{equation}
for $f\in\Stwoq$.  This can be checked on a vector space basis
  for the algebra $\Stwoq$, a basis which can be taken as $X(m) (\Bz)^n$ for
  $m\in \integers$, $n \in\naturals$ with $X(m) = (\Bp)^m$ for
  $m\geq 0$ and $X(m) = (\Bm)^{-m}$ for $m < 0$ (cf. \cite{mnw1991}).

\subsection{A left covariant calculus on $\Stwoq$} 
Since the element $K$ acts (on the left) as the identity on $\Stwoq$, the differential
\eqref{df3} when restricted to $f\in\Stwoq$ becomes
\begin{align}\label{diff-s2l}
d f  =  (\Xm \rtr f) \,\omega_{-} + (\Xp \rtr f) \,\omega_{+} .
\end{align}
Note that $\Xpm \rtr f \notin \Stwoq$. In particular one finds
\begin{align*}
d\Bp &=  q \, (a^\ast)^{2} \, \omegap + c^{2} \, \omegam ,  \\
d\Bm &=  - q^{2}\, (c^\ast)^{2} \, \omegap - q^{-1} \, a^{2} \, \omegam,  \\
d\Bz & =  c^\ast a^\ast\, \omegap -q^{-1} c a  \, \omegam 
\end{align*}
which can be inverted to yield
\begin{align*}
\omegap & = q^{-1} a^2 \, d\Bp -q^2 c^2 \, d\Bm + (1+q^2) ac \, d\Bz  \\
\omegam & = (c^\ast)^{2} \, d\Bp - q (a^\ast)^{2} \, d\Bm - (1+q^2) c^\ast a^\ast \, d\Bz,
\end{align*}
implying that the differential in \eqref{diff-s2l} can be expressed as
\begin{align*}
d f & = \paraa{ q^{-1} (\Xp \rtr f) \, a^2 + (\Xm \rtr f) \, (c^\ast)^{2}}d\Bp \\
& \qquad - 
\paraa{ q^2 (\Xp \rtr f) \, + c^2 q (\Xm \rtr f) \, (a^\ast)^{2}} d\Bm \\ 
& \qquad + (1+q^2)\paraa{ (\Xp \rtr f) \, ac - (\Xm \rtr f) \, c^\ast a^\ast}d\Bz .
\end{align*}
From this expression one finds that the differential $d$ on $\Stwoq$ 
can be written in terms of the right acting operators $Y_a$.
\begin{lemma}
For $f\in\Stwoq$, the differential in \eqref{diff-s2l} can be written as
\begin{equation}\label{exd-q}
d f = (f \ltr V_+) \, d\Bp + (f\ltr V_-) \,\, d\Bm + (f \ltr V_0) \, d\Bz  
\end{equation}
where
\begin{align*}
V_+ & = \Yp\paraa{ 1 - q^{-2}(1+q^2) \Bz} q^{-1} - \Yz \,\, \Bm  \frac{q^{-2}(1+q^6)}{1+q^4} 
+ \Yz^2 \,\, \Bm\frac{1-q^2}{(1+q^2)(1+q^4)} 
\\~~\\
V_- & = - \Ym\paraa{ 1 - q^2(1+q^2) \Bz} q + \Yz \,\, \Bp  \frac{q^{-2}(1+q^6)}{1+q^4} 
- \Yz^2 \,\, \Bp \frac{1-q^2}{(1+q^2)(1+q^4)}
\\~~\\
V_0 & = \big( \Yp \,\, \Bp \, q^{-1} -  \Ym \,\, \Bm \, q \big) (1+q^2) + \Yz \,\, \Bz \frac{(1-q^4)(1+q^6)}{1+q^4}   
- \Yz^2 \,\, \Bz \frac{1-q^2}{1+q^4} .
\end{align*}
\end{lemma}
\begin{proof}
By acting on the vector space basis $X(m)(\Bz)^n$ (as introduced previously), one explicitly checks the equality of \eqref{diff-s2l} and \eqref{exd-q} via a tedious but straightforward computation.
\end{proof}

\begin{remark}
  When $q=1$ the derivative \eqref{exd-q} reduces to
  \begin{equation}\label{classdf}
    \begin{split}
      d f &= 2 \paraa{ ( f \ltr \Yp) \,\, \Bp  - (f \ltr \Ym) \,\, \Bm} d\Bz \\
      & \quad + \paraa{ ( f \ltr \Yp)  \,\, (1-2 \Bz) - ( f \ltr \Yz)  \,\, \Bm } d\Bp \\
      & \quad + \paraa{ - ( f \ltr \Ym) \,\, (1-2 \Bz) +  (f \ltr \Yz) \,\, \Bp} d\Bm \, .       
    \end{split}
  \end{equation}
Classically, the vector field $X_a$ are the left invariant vector fields on $S^3 = SU(2)$ 
with dual left invariant forms $\omega_a$. Thus they do not project to vector fields on the base space $S^2$  
with commuting coordinates $(\Bp, \Bm, \Bz)$ and relation $\Bp \Bm = \Bz (1-\Bz)$: 
$X_a \rtr f $ is not a function on $S^2$ even when $f$ is. On the other hand, the vector fields
$Y_a$ are the right invariant vector fields on $SU(2)$ and thus they project to vector fields on $S^2$, where they are not independent any longer and are related by  
\begin{align} \label{classrel-rvf}
2 ( \Bp \Yp + \Bm \Ym ) + (1-2 \Bz) \Yz = 0 ,
\end{align}
which is just the relation to which \eqref{rel-rvf} reduces when $q=1$.

By changing coordinates $B_0 = \tfrac{1}{2}(1-x)$ so that the radius condition for $S^2$ is written as $r^2 = 4 \Bp \Bm + x^2$, 
the exterior derivative operator in \eqref{classdf} becomes
\begin{align*} 
d f = \partial_x f \, d x + \partial_+ f \, d \Bp + \partial_- f \, d \Bm 
- (\Delta f) \, ( x\, dx + 2 \Bm \, d\Bp + 2 \Bp \, d\Bm ) 
\end{align*}
where $\Delta = x \, \partial_x + \Bp \, \partial_+ + \Bm \, \partial_-$ is the Euler (dilatation) vector field. One then computes 
$d r^2 = 2 ( 1 - r^2 ) ( x\, dx + 2 \Bm \, d\Bp + 2 \Bp \, d \Bm )$, which vanishes when restricting to $S^2$: $r^2-1=0$. 

The form \eqref{diff-s2l} of the differential that uses left invariant vector fields and forms 
can be seen as identifying the cotangent bundle of $S^2$ with the direct sum of the line bundles of `charge' $\pm 2$, that is   
$\Omega^{1}(S^2) \simeq  \L_{-2} \omega_{-} \oplus \L_{+2} \omega_{+}$. This identification can be used also for the quantum 
sphere $\Stwoq$ with the line bundles defined as in \eqref{linbun} below. 
\end{remark}

\subsection{Connections on projective modules over $\Stwoq$}
The definitions of $\sigma$-modules and
$q$-affine connections apply equally well to the algebra $\Stwoq$.
Note that the right action of $K$ preserve $\Stwoq$
\begin{align*}
  \Bz\ltr K = \Bz\qquad
  \Bp\ltr K = q\Bp\qquad
  \Bm\ltr K = q^{-1}\Bm.
\end{align*}
implying that $\sigma_a,\sigmas_a$ leave $\Stwoq$ invariant.

In this section, we will construct $q$-affine connections on a class
of projective modules over $\Stwoq$
(cf. \cite{bm:line.bundles.quantum.spheres,hm:projective.monopole,l:twisted.sigma.model}). For
$n\geq 0$ and $\mu=0,1,\ldots,n$, let
$(\Psi_n)_\mu,(\Phi_n)_\mu\in\Stq$ be given as
\begin{align*}
  (\Phi_n)_\mu = \sqrt{\alpha_{n\mu}}c^{n-\mu}a^{\mu} \qquad
  \quad
 (\Psi_n)_\mu = \sqrt{\beta_{n\mu}}(c^\ast)^{\mu}(a^\ast)^{n-\mu}
\end{align*}
with
\begin{align*}
  \alpha_{n\mu}=\prod_{k=0}^{n-\mu-1}\frac{1-q^{2(n-k)}}{1-q^{2(k+1)}}\qquad\qquad
  \beta_{n\mu}=q^{2\mu}\prod_{k=0}^{\mu-1}\frac{1-q^{-2(n-k)}}{1-q^{-2(k+1)}}.
\end{align*}
It is straight-forward to check that
\begin{align*}
  \sum_{\mu=0}^{n}(\Phi_n)_{\mu}^\ast(\Phi_n)_{\mu}=
  \sum_{\mu=0}^{n}(\Psi_n)_{\mu}^\ast(\Psi_n)_{\mu}= \mid,
\end{align*}
implying that
\begin{align*}
  &(p_{n})_{\mu}^{\nu} = (\Psi_n)_\mu(\Psi_n)^\ast_{\nu}
  =\sqrt{\beta_{n\mu}\beta_{n\nu}}(c^\ast)^{\mu}(a^\ast)^{n-\mu}a^{n-\nu}c^{\nu}\\
  &(p_{-n})_{\mu}^{\nu} = (\Phi_n)_\mu(\Phi_n)^\ast_{\nu}
  =\sqrt{\alpha_{n\mu}\alpha_{n\nu}}c^{n-\mu}a^\mu(a^\ast)^\nu (c^\ast)^{n-\nu}
\end{align*}
satisfy $p_n^2=p_n$ and $p_{-n}^2=p_{-n}$. Moreover, it is easy to see that
$(p_{n})_{\mu}^\nu,(p_{-n})_{\mu}^\nu\in\Stwoq$, which implies that
\begin{align*}
  M_n =
  \begin{cases}
    p_n(\Stwoq)^{n+1}&\text{if }n\geq 0\\
    p_{-|n|}(\Stwoq)^{|n|+1}&\text{if }n<0
  \end{cases}
\end{align*}
are finitely generated projective $\Stwoq$-modules for
$n\in\integers$. Let us recall that these modules are isomorphic to
the components in a (vector space) decomposition of $\Stq$
\begin{align*}
  \Stq = \oplus_{n\in\integers}\L_n
\end{align*}
with
\begin{align} \label{linbun}
  \L_n = \{f\in\Stq:\alpha_z(f) = \bar{z}^nf\},
\end{align}
and it follows that $\L_0 = \Stwoq$, as well as
$\L_n\L_m\subseteq\L_{n+m}$. For $f\in \Stwoq$ and $f_n\in\L_n$ one
has
\begin{align*}
  \alpha_z(ff_n) = \alpha_z(f)\alpha_z(f_n) = \bar{z}^nff_n,
\end{align*}
which implies that $\L_n$ is a left $\Stwoq$-module. Furthermore, it
is easy to see that the right action of $\Uqsu$ leaves each $\L_n$
invariant. Let $\{e_\mu\}_{\mu=0}^n$ be a basis of $(\Stwoq)^{n+1}$ and
let $\phi_n^0:(\Stwoq)^{n+1}\to\L_n$ be defined as
\begin{align*}
  \phi_n^0(m^\mu e_{\mu}) =
  \begin{cases}
    m^\mu (\Psi_n)_\mu &n\geq 0\\
    m^\mu(\Phi_{|n|})_\mu &n<0
  \end{cases},  
\end{align*}
and we note that $\phi^0_n(m)\in\L_n$ since $(\Psi_n)_\mu\in\L_n$ and $(\Phi_n)_\mu\in L_{-n}$ (for $n\geq 0$).

\begin{lemma}
  Let $m\in(\Stwoq)^{n+1}$. If $p_n(m)=0$ then $\phi^0_n(m)=0$.
\end{lemma}

\begin{proof}
  If $p_n(m)=0$ then (for $n\geq 0$)
  \begin{align*}
    &\sum_{\mu=0}^nm^\mu (p_n)_\mu^\nu=0\quad\forall\,\nu\implies
      \sum_{\mu,\nu=0}^nm^\mu (p_n)_\mu^\nu(\Psi_n)_{\nu}=0\implies\\
    &\parab{\sum_{\mu=0}^nm^\mu(\Psi_n)_\mu}\parab{\sum_{\nu=0}^n(\Psi_n)_{\nu}^\ast(\Psi_n)_{\nu}}=0
      \implies\sum_{\mu=0}^nm^\mu(\Psi_n)_\mu=0
  \end{align*}
  which is equivalent to $\phi^0_n(m)=0$. The proof for $n<0$ is analogous.
\end{proof}

\noindent
The above result implies that $\phi_n^0$ descends to a module
homomorphism $\phi_n:M_n\to\L_n$, and one can show that $\phi_n$ is in
fact an isomorphism. To simplify the
presentation, let us in the following assume that $n\geq 0$.  As
generators of $M_n$ one can choose $\eh_\mu=p_n(e_\mu)$ for
$\mu=0,1,\ldots,n$, and one notes that
\begin{align*}
  \phi_n(\eh_\mu) = (p_n)_\mu^\nu(\Psi_n)_\nu = (\Psi_n)_{\mu},
\end{align*}
implying that $\{(\Psi_n)_\mu\}_{\mu=0}^n$ generates $\L_n$.  Since
the $\Stwoq$-module $\L_n$ is a subset of $\Stq$ which is also invariant
under the right action of $\Uqsu$, it is naturally a $\sigma$-module
with respect to the right action of $\sigma_a$ and
$\sigmas_a$; one finds that
\begin{align*}
  (\Psi_n)_{\mu}\ltr K = q^{\tfrac{1}{2}(n-2\mu)}(\Psi_n)_{\mu}\qquad
  (\Psi_n)_{\mu}^\ast\ltr K = q^{-\tfrac{1}{2}(n-2\mu)}(\Psi_n)_{\mu}^\ast
\end{align*}
giving
\begin{alignat*}{2}
  &\sigmapm\paraa{(\Psi_n)_\mu} = q^{n-2\mu}(\Psi_n)_\mu &\qquad
  &\sigmaz\paraa{(\Psi_n)_\mu} = q^{2(n-2\mu)}(\Psi_n)_\mu\\
  &\sigmaspm\paraa{(\Psi_n)_\mu} = q^{-(n-2\mu)}(\Psi_n)_\mu &\qquad
  &\sigmasz\paraa{(\Psi_n)_\mu} = q^{-2(n-2\mu)}(\Psi_n)_\mu.
\end{alignat*}
Correspondingly, we would like to define a $\sigma$-module structure
on $M_n$ such that $\phi_n$ is a morphism of $\sigma$-modules. To this
end, we start by introducing a $\sigma$-module structure on
$(\Stwoq)^{n+1}$ in analogy with
Proposition~\ref{prop:q.affine.projective}; namely, starting from
\begin{alignat*}{2}
  &\sigmah^0_\pm(e_\mu) = q^{n-2\mu}e_\mu&\qquad
  &\sigmah^0_z(e_\mu) = q^{2(n-2\mu)}e_\mu\\
  &(\sigmah^0_\pm)^\ast(e_\mu) = q^{-(n-2\mu)}e_\mu&\qquad
  &(\sigmah^0_z)^\ast(e_\mu) = q^{-2(n-2\mu)}e_\mu
\end{alignat*}
we set $\sigmah_a=p\circ\sigmah_a^0:M_n\to M_n$ and
$\sigmah_a^\ast=p\circ(\sigmah_a^0)^\ast:M_n\to M_n$. Since
\begin{align*}
  (p_n)_\mu^\nu\ltr K = \paraa{(\Psi_n)_\mu(\Psi_n)_\nu}\ltr K
  =q^{-(\mu-\nu)}(p_n)_\mu^\nu
\end{align*}
one finds that
\begin{align*}
  \sigmah_\pm(\eh_\mu) &=(p_n\circ\sigmah^0_\pm)\paraa{(p_n)_\mu^\nu e_\nu}=
  \sigma_\pm\paraa{(p_n)_\mu^\nu}(p_n\circ\sigmah^0_\pm)(e_\nu)\\
                       &=q^{-2(\mu-\nu)}q^{n-2\nu}p_n(e_{\nu})
                         =q^{n-2\mu}\eh_{\mu}\\
  \sigmahz(\eh_\mu) &= q^{2(n-2\mu)}\eh_\mu
\end{align*}
(and similarly for $\sigmah_a^\ast$) implying that $\phi_n$ is a
$\sigma$-module isomorphism. Finally, let us now turn to the question
of finding $q$-affine connections on $M_n$ that are compatible with a
given hermitian form. Thus, assume that $h$ is a hermitian form on
$(\Stwoq)^{n+1}$ for which $p_n$ is an orthogonal
projection. According to Proposition~\ref{prop:nabla.comp.metric},
every $q$-affine connection $\nabla_0$ on $(\Stwoq)^{n+1}$ may be written in the
form $\nabla^0_{X_a}e_\mu=\Gamma_{a\mu}^\nu e_\nu$ with 
\begin{align}\label{eq:Gammat.S2q}
  \begin{split}
    &\Gammat_{+\mu,\nu} = \tfrac{1}{2}\Xp(h_{\mu\nu})+K(\alpha_{\mu\nu})+iK(\beta_{\mu\nu})\\
    &\Gammat_{-\mu,\nu} = \tfrac{1}{2}\Xm(h_{\mu\nu})+K(\alpha_{\mu\nu})-iK(\beta_{\mu\nu})\\
    &\Gammat_{z\mu,\nu} = \tfrac{1}{2}\Xz(h_{\mu\nu})+K^2(\rho_{\mu\nu}),
  \end{split}
\end{align}
for an arbitrary choice of hermitian
$\alpha,\beta,\rho\in\Mat_{n+1}(\Stwoq)$ and, furthermore,
Proposition~\ref{prop:q.affine.projective} shows that
$\nabla=p_n\circ\nabla^0$ is a $q$-affine connection on $M_n$,
compatible with the restriction of $h$ to $M_n$. For the generators
$\{\eh_\mu\}_{\mu=0}^n$ one finds that
\begin{align*}
  \nabla_{\Xp}\eh_\mu &= (p_n\circ\nabla^0)\paraa{(p_n)_\mu^\nu e_\nu}
  =p\paraa{(p_n)_\mu^\nu\Gamma_{+\nu}^\kappa e_\kappa+\Xp\paraa{(p_n)_\mu^\nu}\sigmah_+(e_\nu)}\\
  &=\parab{(p_n)_\mu^\nu\Gamma_{+\nu}^\kappa+q^{n-2\mu}\Xp\paraa{(p_n)_\mu^\kappa}}\eh_\kappa\\
  \nabla_{\Xm}\eh_\mu &=\parab{(p_n)_\mu^\nu\Gamma_{-\nu}^\kappa+q^{n-2\mu}\Xm\paraa{(p_n)_\mu^\kappa}}\eh_\kappa\\
  \nabla_{\Xz}\eh_\mu &=\parab{(p_n)_\mu^\nu\Gamma_{z\nu}^\kappa+q^{2(n-2\mu)}\Xz\paraa{(p_n)_\mu^\kappa}}\eh_\kappa.
\end{align*}
In the current case, one finds that
\begin{align*}
  &\htilde_{\pm\mu\nu} = h\paraa{(\sigmahpm^0)^\ast(e_\mu),e_\nu} =
    q^{2\mu-n}h_{\mu\nu}\implies
  \htilde_{\pm}^{\mu\nu}=h^{\mu\nu}q^{n-2\nu}\\
  &\htilde_{z\mu\nu} = h\paraa{(\sigmahz^0)^\ast(e_\mu),e_\nu}
    =q^{2(2\mu-n)}h_{\mu\nu}\implies
    \htilde_{z}^{\mu\nu} = h^{\mu\nu}q^{2(n-2\nu)}
\end{align*}
giving
\begin{align*}
  &\Gamma_{\pm\mu}^\nu = \Gammat_{\pm\mu,\rho}
    \sigmapm\paraa{\htilde_{\pm}^{\rho\mu}}
  =q^{2\nu-n}\Gammat_{\pm\mu,\rho}K^2(h^{\rho\nu})\\
  &\Gamma_{z\mu}^\nu = \Gammat_{z\mu,\rho}
    \sigmaz(\htilde_{z}^{\rho\nu})
    =q^{2(2\nu-m)}\Gammat_{z\mu,\rho}K^4(h^{\rho\nu})
\end{align*}
Thus, for the choice $\alpha=\beta=\rho=0$ in \eqref{eq:Gammat.S2q},
one obtains the formulas
\begin{align*}
  \nabla_{\Xpm}\eh_\mu
  &=\parac{\tfrac{1}{2}q^{2\kappa-n}(p_n)_\mu^\nu\Xpm(h_{\nu\rho})K^2(h^{\rho\kappa})
    +q^{n-2\mu}\Xpm\paraa{(p_n)_\mu^\kappa}}\eh_\kappa\\
  \nabla_{\Xz}\eh_\mu
  &=\parac{\tfrac{1}{2}q^{2(2\kappa-n)}(p_n)_\mu^\nu\Xz(h_{\nu\rho})K^4(h^{\rho\kappa})
    +q^{2(n-2\mu)}\Xz\paraa{(p_n)_\mu^\kappa}}\eh_\kappa.
\end{align*}

\section{Computation of Christoffel symbols}\label{sec:appendix}

\noindent
In this section we work out an explicit expression for a metric
and torsion free connection on $\OmegaStq$, as presented in Section \ref{se:4}, for the case
when the metric is invariant by $K$, that is $K(h_{ab})=h_{ab}$, and
$K(\omega_a)=\omega_a$ for $a,b=1,2,3$. 
As previously noted, for such a metric, $\Gammat_{ab,c} = \Gamma_{ab}^ph_{pc}$
and the torsion free equations become (cf. \eqref{ga1}-\eqref{ga3}):
\begin{align*}
  &\Gammat_{-+,a}-q^2\Gammat_{+-,a} = h_{za}\\
  q^2&\Gammat_{z-,a}-q^{-2}\Gammat_{-z,a} = (1+q^2)h_{-a}\\
  q^{2}&\Gammat_{+z,a}-q^{-2}\Gammat_{z+,a} = (1+q^2)h_{+a} .
\end{align*}
These can be solved as
\begin{align}
  &\Gammat_{-+,a} = \tfrac{1}{2}h_{za}+qK(\gone_a)\\
  &\Gammat_{+-,a} = -\tfrac{1}{2}q^{-2}h_{za}+q^{-1}K(\gone_a)\\
  &\Gammat_{z-,a} = \tfrac{1}{2}q^{-2}(1+q^2)h_{-a}+q^{-2}K(\gtwo_a)\\
  &\Gammat_{-z,a} = -\tfrac{1}{2}q^{2}(1+q^2)h_{-a}+q^{2}K(\gtwo_a)\\
  &\Gammat_{+z,a} = \tfrac{1}{2}q^{-2}(1+q^2)h_{+a}+q^{-2}K(\gthree_a)\\
  &\Gammat_{z+,a} = -\tfrac{1}{2}q^{2}(1+q^2)h_{+a}+q^{2}K(\gthree_a)
\end{align}
for arbitrary $\gone_a,\gtwo_a,\gthree_a\in\Stq$. On the other hand,
Proposition~\ref{prop:nabla.comp.metric} gives general expressions
for $\Gammat_{ab,c}$ for a metric connection, and combining the two
results yields the following set of equations to be solved:
\begin{align}
  &\tfrac{1}{2}\Xm(h_{+a})+K(\gamma_{a+}^\ast) = \tfrac{1}{2}h_{za}+qK(\gone_a)\\
  &\tfrac{1}{2}\Xp(h_{-a})+K(\gamma_{-a}) = -\tfrac{1}{2}q^{-2}h_{za}+q^{-1}K(\gone_a)\\
  &\tfrac{1}{2}\Xz(h_{-a}) + K^2(\rho_{-a}) = \tfrac{1}{2}q^{-2}(1+q^2)h_{-a}+q^{-2}K(\gtwo_a)\\
  &\tfrac{1}{2}\Xm(h_{za}) + K(\gamma_{az}^\ast) = -\tfrac{1}{2}q^{2}(1+q^2)h_{-a}+q^{2}K(\gtwo_a)\\
  &\tfrac{1}{2}\Xp(h_{za})+K(\gamma_{za}) = \tfrac{1}{2}q^{-2}(1+q^2)h_{+a}+q^{-2}K(\gthree_a)\\
  &\tfrac{1}{2}\Xz(h_{+a})+K^2(\rho_{+a}) = -\tfrac{1}{2}q^{2}(1+q^2)h_{+a}+q^{2}K(\gthree_a)  
\end{align}
where $\gamma_{ab}=\alpha_{ab}+i\beta_{ab}$. These equations can be rewritten as
\begin{align}
  \gamma_{-a} &= -\thalf K^{-1}\Xp(h_{-a})-\thalf q^{-2}h_{za}+q^{-1}\gone_a
                \label{eq:lc.gma}\\
  \gamma_{a+} &= \thalf K^{-1}\Xp(h_{a+})+\thalf h_{az}+q(\gone_a)^\ast
                \label{eq:lc.gap}\\
  \rho_{-a} &= -\thalf K^{-2}\Xz(h_{-a})+\thalf q^{-2}(1+q^2)h_{-a}+q^{-2}K^{-1}(\gtwo_a)\label{eq:lc.rhoma}\\
  \gamma_{az} &= \thalf K^{-1}\Xp(h_{az})-\thalf q^2(1+q^2)h_{a-}+q^2(\gtwo_a)^\ast\label{eq:lc.gaz}\\
  \gamma_{za} &= -\thalf K^{-1}\Xp(h_{za})+\thalf q^{-2}(1+q^2)h_{+a}+q^{-2}\gthree_a\label{eq:lc.gza}\\
  \rho_{+a} &= -\thalf K^{-2}\Xz(h_{+a})-\thalf q^2(1+q^2)h_{+a}+q^2K^{-1}(\gthree_a).\label{eq:lc.rhopa}
\end{align}
When solving for $\gamma_{ab}$ and $\rho_{ab}$ from the above
equations, one finds several ambiguities and constraints. Namely,
there are multiple expressions for $\gamma_{-+}$, $\gamma_{-z}$,
$\gamma_{z+}$ and $\gamma_{zz}$, and we require (from
Proposition~\ref{prop:nabla.comp.metric}) that
$\rho_{++}^\ast=\rho_{++}$, $\rho_{--}^\ast=\rho_{--}$ and
$\rho_{+-}^\ast=\rho_{-+}$. Let us consider these constraints one by
one.

\medskip

\noindent
\underline{$\gamma_{-+}$}: Equations \eqref{eq:lc.gma} and
\eqref{eq:lc.gap} give two expressions for $\gamma_{-+}$:
\begin{align*}
  &\gamma_{-+} = -\thalf K^{-1}\Xp(h_{-+})-\thalf q^{-2}h_{z+}+q^{-1}\gone_+\\
  &\gamma_{-+} = \thalf K^{-1}\Xp(h_{-+})+\thalf h_{-z}+q(\gone_-)^\ast
\end{align*}
which coincide if
\begin{align*}
  &\gone_+ = \thalf q K^{-1}\Xp(h_{-+})+\thalf q^{-1}h_{z+}+q\tau_1\\
  &\gone_- = -\thalf q^{-1}K^{-1}\Xm(h_{+-})-\thalf q^{-1}h_{z-}+q^{-1}\tau_1^\ast
\end{align*}
for arbitrary $\tau_1\in\Stq$, giving $\gamma_{-+}=\tau_1$.

\medskip

\noindent
\underline{$\gamma_{-z}$}: Equations \eqref{eq:lc.gma} and
\eqref{eq:lc.gaz} give two expressions for $\gamma_{-z}$:
\begin{align*}
  &\gamma_{-z} = -\thalf K^{-1}\Xp(h_{-z})-\thalf q^{-2}h_{zz}+q^{-1}\gone_z\\
  &\gamma_{-z} = \thalf K^{-1}\Xp(h_{-z})-\thalf q^2(1+q^2)h_{--}+q^{2}(\gtwo_-)^\ast
\end{align*}
which coincide if
\begin{align}
  &\gone_z = \thalf q K^{-1}\Xp(h_{-z})+\thalf q^{-1}h_{zz}+q\tau_2\label{eq:lc.gonez.1}\\
  &\gtwo_- = \thalf q^{-2} K^{-1}\Xm(h_{z-})+\thalf(1+q^2)h_{--}+q^{-2}\tau_2^\ast\notag
\end{align}
for arbitrary $\tau_2\in\Stq$, giving $\gamma_{-z}=\tau_2$.

\medskip

\noindent
\underline{$\gamma_{z+}$}: Equations \eqref{eq:lc.gap} and
\eqref{eq:lc.gza} give two expressions for $\gamma_{z+}$:
\begin{align*}
  &\gamma_{z+} = \thalf K^{-1}\Xp(h_{z+})+\thalf h_{zz}+q(\gone_z)^\ast\\
  &\gamma_{z+} = -\thalf K^{-1}\Xp(h_{z+})+\thalf q^{-2}(1+q^2)h_{++}
    +q^{-2}\gthree_+
\end{align*}
which coincide if
\begin{align}
  &\gone_z = \thalf q^{-1}K^{-1}\Xm(h_{+z})-\thalf q^{-1}h_{zz}+q^{-1}\tau_3^\ast\label{eq:lc.gonez.2}\\
  &\gthree_+ = \thalf q^{2}K^{-1}\Xp(h_{z+})-\thalf(1+q^2)h_{++}+q^2\tau_3\notag
\end{align}
for arbitrary $\tau_3\in\Stq$, giving $\gamma_{z+}=\tau_3$.

\medskip

\noindent
\underline{$\gamma_{zz}$}: Equations \eqref{eq:lc.gaz} and
\eqref{eq:lc.gza} give two expressions for $\gamma_{zz}$:
\begin{align*}
  &\gamma_{zz} = \thalf K^{-1}\Xp(h_{zz})-\thalf q^{2}(1+q^2)h_{z-}+q^{2}(\gtwo_z)^\ast\\
  &\gamma_{zz} = -\thalf K^{-1}\Xp(h_{zz})+\thalf q^{-2}(1+q^2)h_{+z}+q^{-2}\gthree_z
\end{align*}
which coincide if
\begin{align*}
  &\gtwo_z = \thalf q^{-2}K^{-1}\Xm(h_{zz})+\thalf(1+q^2)h_{-z}+q^{-2}\tau_4^\ast\\
  &\gthree_z = \thalf q^2 K^{-1}\Xp(h_{zz})-\thalf(1+q^2)h_{+z}+q^{2}\tau_4
\end{align*}
for arbitrary $\tau_4\in\Stq$, giving $\gamma_{zz}=\tau_4$.

\medskip

\noindent
The above considerations determine $\gone_+$, $\gone_-$, $\gtwo_-$,
$\gthree_+$, $\gtwo_z$, $\gthree_z$, $\gone_z$, but give two
expressions for $\gone_z$; equating them gives the following.

\medskip

\noindent
\underline{$\gone_z$}: Equations \eqref{eq:lc.gonez.1} and
\eqref{eq:lc.gonez.2} give two expressions for $\gone_z$:
\begin{align*}
  &\gone_z = \thalf q K^{-1}\Xp(h_{-z})+\thalf q^{-1}h_{zz}+q\tau_2\\
  &\gone_z = \thalf q^{-1}K^{-1}\Xm(h_{+z})-\thalf q^{-1}h_{zz}+q^{-1}\tau_3^\ast
\end{align*}
which coincide if
\begin{align*}
  &\tau_2 = -\thalf K^{-1}\Xp(h_{-z})-\thalf q^{-2}h_{zz}+q^{-1}\mu_1\\
  &\tau_3 = \thalf K^{-1}\Xp(h_{z+})+\thalf h_{zz}+q\mu_1^\ast
\end{align*}
for arbitrary $\mu_1\in\Stq$, giving $\gone_z=\mu_1$ and
\begin{align*}
  &\gtwo_- = q^{-2} K^{-1}\Xm(h_{z-})+\thalf(1+q^2)h_{--}
    -\thalf q^{-4}h_{zz}+q^{-3}\mu_1^\ast\\
  &\gthree_+ = q^{2}K^{-1}\Xp(h_{z+})-\thalf(1+q^2)h_{++}
    +\thalf q^2h_{zz}+q^3\mu_1^\ast.
\end{align*}

\medskip

\noindent Let us summarize what we have obtained so far:
\begin{align}
  &\gone_+ = \thalf q K^{-1}\Xp(h_{-+})+\thalf q^{-1}h_{z+}+q\tau_1\\  
  &\gone_- = -\thalf q^{-1}K^{-1}\Xm(h_{+-})-\thalf q^{-1}h_{z-}+q^{-1}\tau_1^\ast\\
  &\gone_z = \mu_1\\
  &\gtwo_- = q^{-2} K^{-1}\Xm(h_{z-})+\thalf(1+q^2)h_{--}
    -\thalf q^{-4}h_{zz}+q^{-3}\mu_1^\ast\label{eq:lc.gtwom}\\
  &\gthree_+ = q^{2}K^{-1}\Xp(h_{z+})-\thalf(1+q^2)h_{++}
    +\thalf q^2h_{zz}+q^3\mu_1^\ast\label{eq:lc.gthreep}\\
  &\gtwo_z = \thalf q^{-2}K^{-1}\Xm(h_{zz})+\thalf(1+q^2)h_{-z}+q^{-2}\tau_4^\ast\\
  &\gthree_z = \thalf q^2 K^{-1}\Xp(h_{zz})-\thalf(1+q^2)h_{+z}+q^{2}\tau_4,
\end{align}
for arbitrary $\mu_1,\tau_1,\tau_4$. It remains to solve the
constraints $\rho_{++}^\ast=\rho_{++}$, $\rho_{--}^\ast=\rho_{--}$ and
$\rho_{+-}^\ast=\rho_{-+}$.

\medskip

\noindent
\underline{$\rho_{++}$}: From equation \eqref{eq:lc.rhopa} one obtains
\begin{align*}
  \rho_{++} = -\thalf q^2(1+q^2)h_{++}+q^2K^{-1}(\gthree_+)
\end{align*}
by using that $\Xz(h_{ab}) = 0$. Requiring $\rho_{++}^\ast = \rho_{++}$ gives
\begin{align*}
  \gthree_+ = K^2\paraa{(\gthree_+)^\ast}
\end{align*}
and inserting \eqref{eq:lc.gthreep} yields
\begin{align*}
  K^2(\mu_1)-\mu_1^\ast = q^{-1}K^{-1}\Xp(h_{z+})
  +q^{-1}K\Xm(h_{+z})
\end{align*}
which is solved by
\begin{align}
  \mu_1 = q^{-1}K^{-1}\Xm(h_{+z})+K^{-1}(f_1)\label{eq:lc.mu1.1}
\end{align}
for arbitrary hermitian $f_1\in\Stq$.

\medskip

\noindent
\underline{$\rho_{--}$}: From equation \eqref{eq:lc.rhoma} one obtains
\begin{align*}
  \rho_{--} = \thalf q^{-2}(1+q^2)h_{--}+q^{-2}K^{-1}(\gtwo_-)
\end{align*}
by using that $\Xz(h_{ab}) = 0$. Requiring $\rho_{--}^\ast = \rho_{--}$ gives
\begin{align*}
  \gtwo_- = K^2\paraa{(\gtwo_-)^\ast}
\end{align*}
and inserting \eqref{eq:lc.gtwom} gives
\begin{align*}
  K^2(\mu_1)-\mu_1^\ast = qK^{-1}\Xm(h_{z-})+qK\Xp(h_{-z})
\end{align*}
which is solved by
\begin{align}
  \mu_1 = qK^{-1}\Xp(h_{-z})+K^{-1}(f_2)\label{eq:lc.mu1.2}
\end{align}
for arbitrary hermitian $f_2\in\Stq$.

Equations \eqref{eq:lc.mu1.1} and
\eqref{eq:lc.mu1.2} give two expressions for $\mu_1$, which coincide
if
\begin{align}
  f_1-f_2 = q\Xp(h_{-z})-q^{-1}\Xm(h_{+z}).\label{f1.m.f2}
\end{align}
Since $f_1-f_2$ is hermitian, a necessary condition for this
equation to have a solution is that
\begin{align*}
  H = q\Xp(h_{-z})-q^{-1}\Xm(h_{+z})
\end{align*}
is hermitian. Using $K^{-1}\Xpm = q^{\mp 1}\Xpm K^{-1}$ and
$K^{-1}(h_{ab})=h_{ab}$, $H=H^\ast$ is equivalent to
\begin{align}\label{eq:nec.lc.metric.cond}
  \Xp(h_{-z}-q^{-4}h_{z+}) = -q^2\Xm(h_{z-}-q^{-4}h_{+z}).
\end{align}
which can also be written as the reality condition
  \begin{align}\label{eq:nec.lc.metric.cond.real}
    \paraa{\Xp(h_{-z}-q^{-4}h_{z+})}^\ast = \Xp(h_{-z}-q^{-4}h_{z+}).
  \end{align}
If the above condition is fulfilled, then \eqref{f1.m.f2} is solved by
\begin{align*}
  &f_1 = \thalf q\Xp(h_{-z}) - \thalf q^{-1}\Xm(h_{+z})+f_0\\
  &f_2 = -\thalf q\Xp(h_{-z}) + \thalf q^{-1}\Xm(h_{+z})+f_0
\end{align*}
for arbitrary hermitian $f_0\in\Stq$. This gives
\begin{align}
  \mu_1 = \thalf qK^{-1}\Xp(h_{-z})+\thalf q^{-1}K^{-1}\Xm(h_{+z})+K^{-1}(f_0).
\end{align}

\medskip

\noindent
\underline{$\rho_{+-}$}: Equations \eqref{eq:lc.rhoma} and
\eqref{eq:lc.rhopa} give
\begin{align*}
  &\rho_{+-} = -\thalf q^2(1+q^2)h_{+-}+q^2K^{-1}\paraa{\gthree_-}\\
  &\rho_{-+} = \thalf q^{-2}(1+q^2)h_{-+} + q^{-2}K^{-1}\paraa{\gtwo_+}
\end{align*}
and requiring $\rho_{+-}^\ast=\rho_{-+}$ yields
\begin{align*}
  -\thalf q^2(1+q^2)h_{-+}+q^2K\paraa{(\gthree_-)^\ast}
  =\thalf q^{-2}(1+q^2)h_{-+} + q^{-2}K^{-1}\paraa{\gtwo_+}
\end{align*}
which is solved by
\begin{align*}
  &\gthree_- = \thalf(1+q^2)h_{+-}+q^{-2}K(\mu_2^\ast) \\
  &\gtwo_+ = -\thalf(1+q^2)h_{-+}+q^2K(\mu_2)
\end{align*}
for arbitrary $\mu_2\in\Stq$.

\medskip

\noindent
Thus, we have obtained the following equations
\begin{align*}
  &\gone_+ = \thalf q K^{-1}\Xp(h_{-+})+\thalf q^{-1}h_{z+}+q\tau_1\\  
  &\gone_- = -\thalf q^{-1}K^{-1}\Xm(h_{+-})-\thalf q^{-1}h_{z-}+q^{-1}\tau_1^\ast\\
  &\gone_z = \thalf qK^{-1}\Xp(h_{-z})+\thalf q^{-1}K^{-1}\Xm(h_{+z})+K^{-1}(f_0)\\
  &\gtwo_+ = -\thalf(1+q^2)h_{-+}+q^2K(\mu_2)\\
  &\gtwo_- = \thalf q^{-2} K^{-1}\Xm(h_{z-})
    -\thalf q^{-4}K^{-1}\Xp(h_{z+})+\thalf(1+q^2)h_{--}
    -\thalf q^{-4}h_{zz}+q^{-3}K(f_0)\\
  &\gtwo_z = \thalf q^{-2}K^{-1}\Xm(h_{zz})+\thalf(1+q^2)h_{-z}+q^{-2}\tau_4^\ast\\
  &\gthree_+ = \thalf q^{2}K^{-1}\Xp(h_{z+})-\thalf q^4K^{-1}\Xm(h_{z-})
    -\thalf(1+q^2)h_{++}+\thalf q^2h_{zz}+q^3K(f_0)\\
  &\gthree_- = \thalf(1+q^2)h_{+-}+q^{-2}K(\mu_2^\ast)\\
  &\gthree_z = \thalf q^2 K^{-1}\Xp(h_{zz})-\thalf(1+q^2)h_{+z}+q^{2}\tau_4  
\end{align*}
giving all 27 Christoffel symbols as
\begin{align*}
  &\Gammat_{++,+} = \Xp(h_{++})-\thalf q^2\Xm(h_{+-})+h_{+z}+q^2K(\tau_1)\\
  &\Gammat_{+-,-} = -\thalf q^{-2}\Xm(h_{+-})-q^{-2}h_{z-}+q^{-2}K(\tau_1^\ast)\\
  &\Gammat_{+z,z} = \thalf\Xp(h_{zz})+K(\tau_4)\\
  &\Gammat_{++,-} = \thalf\Xp(h_{+-})+K(\gamma_{+-})\\
  &\Gammat_{+-,+} = \thalf\Xp(h_{-+})+K(\tau_1)\\
  &\Gammat_{+z,+} = \thalf\Xp(h_{z+})-\thalf q^2\Xm(h_{z-})
    +\thalf h_{zz}+qK^2(f_0)\\
  &\Gammat_{++,z} = \Xp(h_{+z})-q^2(1+q^2)h_{+-}+q^4\mu_2^\ast\\
  &\Gammat_{+z,-} = q^{-2}(1+q^2)h_{+-}+q^{-4}K^2(\mu_2^\ast)\\
  &\Gammat_{+-,z} = \thalf\Xp(h_{-z})+\thalf q^{-2}\Xm(h_{+z})
    -\thalf q^{-2}h_{zz}+q^{-1}f_0
\end{align*}

\begin{align*}
  &\Gammat_{-+,+} = \thalf q^2\Xp(h_{-+})+h_{z+}+q^2K(\tau_1^\ast)\\
  &\Gammat_{--,-} = \Xm(h_{--})+\thalf q^{-2}\Xp(h_{-+})
    -q^{-2}h_{-z}+q^{-2}K(\tau_1)\\
  &\Gammat_{-z,z} = \thalf\Xm(h_{zz})+K(\tau_4^\ast)\\
  &\Gammat_{-+,-} = \thalf\Xm(h_{+-})+K(\tau_1^\ast)\\
  &\Gammat_{--,+} = \thalf\Xm(h_{-+})+K(\gamma_{+-}^\ast)\\
  &\Gammat_{-+,z} = \thalf\Xm(h_{+z})+\thalf q^{2}\Xp(h_{-z})+\thalf h_{zz}+qf_0\\
  &\Gammat_{-z,+} = -q^2(1+q^2)h_{-+}+q^{-1}K^2(\mu_2)\\
  &\Gammat_{--,z} = \Xm(h_{-z})+q^2(1+q^2)h_{-+}+q^{-4}\mu_2\\
  &\Gammat_{-z,-} = \thalf\Xm(h_{z-})-\thalf q^{-2}\Xp(h_{z+})
    -\thalf q^{-2}h_{zz}+q^{-1}K^2(f_0)
\end{align*}

\begin{align*}
  &\Gammat_{z+,+} = \thalf q^4\Xp(h_{z+})-\thalf q^6\Xm(h_{z-})-q^2(1+q^2)h_{++}
    +\thalf q^4h_{zz}+q^5K^2(f_0)\\
  &\Gammat_{z-,-} = \thalf q^{-4}\Xm(h_{z-})-\thalf q^{-6}\Xp(h_{z+})
    +q^{-2}(1+q^2)h_{--}-\thalf q^{-6}h_{zz}+q^{-5}K^2(f_0)\\
  &\Gammat_{zz,z} = K^2(\rho_{zz})\\
  &\Gammat_{z+,-} = K^2(\mu_2^\ast)\\
  &\Gammat_{z-,+} = K^2(\mu_2)\\
  &\Gammat_{z+,z} = \thalf q^4\Xp(h_{zz})-q^2(1+q^2)h_{+z}+q^4K(\tau_4)\\
  &\Gammat_{zz,+} = -\thalf q^2\Xm(h_{zz})-q^2(1+q^2)h_{z+}+q^4K^3(\tau_4^\ast)\\
  &\Gammat_{z-,z} = \thalf q^{-4}\Xm(h_{zz})+q^{-2}(1+q^2)h_{-z}
    +q^{-4}K(\tau_4^\ast)\\
  &\Gammat_{zz,-} = -\thalf q^{-2}\Xp(h_{zz})+q^{-2}(1+q^2)h_{z-}
    +q^{-4}K^3(\tau_4)   
\end{align*}
depending on the parameters
$\tau_1,\tau_4,\mu_2,\gamma_{+-},\rho_{zz},f_0\in\Stq$ (with
$f_0^\ast=f_0$).

\vspace{4mm}

\noindent
{\bf Acknowledgements:} 
The paper is partially supported by INFN-Trieste. 
GL is supported by INFN, Iniziativa Specifica GAST,
and by INDAM - GNSAGA. JA is supported by grant 2017-03710 from the
Swedish Research Council.  
 Furthermore, JA would like to thank the Department of Mathematics 
 and Geosciences, University of Trieste for hospitality.

\bibliographystyle{alpha}

\end{document}